\documentclass[preprint,3p,11pt]{elsarticle}

\usepackage{fleqn}
\usepackage{color}
\usepackage{amssymb}
\usepackage{amsthm}
\usepackage{amsmath}
\usepackage{hyperref}
\usepackage{amsfonts}
\usepackage{enumerate}
\biboptions{sort&compress,square,numbers}

\usepackage{subcaption}

\newlength{\figwidth}
\newlength{\figheight}
\numberwithin{equation}{section}

\numberwithin{equation}{section}
\newtheorem{theorem}{Theorem}[section]
\newtheorem{proposition}[theorem]{Proposition}
\newtheorem{corollary}[theorem]{Corollary}

\theoremstyle{remark}
\newtheorem{example}{Example}[section]
\newtheorem{remark}[theorem]{Remark}

\usepackage{mathabx}

\newcommand{\E}{\mathcal{E}}
\newcommand{\B}{\mathcal{B}}
\renewcommand{\Pr}{\mathbb{P}}
\newcommand{\dive}{\operatorname{div}}
\newcommand{\Leb}{\operatorname{Leb}}
\newcommand{\Id}{\operatorname{Id}}

\begin{document}

\begin{frontmatter}

\title{Randomly switching evolution equations\tnoteref{t1}}
\tnotetext[t1]{This research was supported in part by the Natural Sciences and Research Council of Canada (NSERC) and  the Polish NCN grant  2017/27/B/ST1/00100.}

\author[a]{Pawe\l{} Klimasara}
\ead{p.klimasara@gmail.com}
\author[b]{Michael C. Mackey\corref{cor1}}
\cortext[cor1]{Corresponding author}
\ead{michael.mackey@mcgill.ca}

\author[c]{Andrzej Tomski}
\ead{andrzej.tomski@us.edu.pl}
\author[c]{Marta Tyran-Kami\'nska}
\ead{mtyran@us.edu.pl}

\address[a]{Chair of Cognitive Science and Mathematical Modelling, University of Information Technology and Management in Rzesz{\'o}w, Sucharskiego 2, 35-225 Rzesz{\'o}w, Poland}
\address[b]{Departments of Physiology, Physics \& Mathematics and Centre for Applied Mathematics in Biology and Medicine, McGill University, 3655 Promenade Sir William Osler, Montreal, QC,
Canada, H3G 1Y6}
\address[c]{Institute of Mathematics, University of Silesia in Katowice, Bankowa 14, 40-007
Katowice, Poland}

\date{\today}

\begin{abstract}

We present an investigation of stochastic evolution in which a family of
evolution equations in $L^1$ are driven by continuous-time Markov processes.
These are examples of so-called piecewise deterministic Markov processes (PDMP's) on the space of integrable functions. We derive equations for the first moment and correlations (of any order) of such processes. We also introduce the mean of the process at large time and describe its behaviour. The results are illustrated by some simple, yet generic,  biological examples characterized by different one-parameter types of  bifurcations.
\end{abstract}


\end{frontmatter}

\section{Introduction}

The theory of piecewise deterministic Markov processes (PDMP's) has generated considerable interest in the scientific community over the past three decades, having been first introduced in \cite{davis84}. From the point of view of modelling in natural sciences the class of PDMP's is a very broad family of stochastic models covering most of the applications, omitting mainly diffusion related phenomena. A recent monograph \cite{rudnickityran17} surveys the applicability of PDMP's to problems in  the biological sciences.

Briefly, a PDMP is a continuous-time Markov process with values in some metric space. The process evolves deterministically between the so-called jump times that form an increasing sequence of random times. Usually deterministic evolution is described by ordinary differential equations (ODE's) inducing dynamical systems (or flows). However, in a PDMP at the jump times one considers a different type of behaviour such that there is an actual jump to a different point in the phase space {\it or}  a change of the dynamics. The latter are referred to as randomly switching dynamical systems  or switching  ODE's.

The class of PDMP's   is widely used in many areas of  science, especially in biology \cite{bressloff17,rudnickityran17}, and these
include the applications of randomly switching dynamical systems.  A prototype of PDMP's is the telegraph process first studied by Goldstein \cite{goldstein} and Kac \cite{Kac} in connection with the telegraph equation where a particle moves on a real line with constant velocity alternating between two opposite values according to a Poisson process. An extension of such a process is a velocity jump process where an individual moves in a space with constant
velocity and at the jump times a new velocity is chosen randomly
\cite{HH,Stroock}. Another example is a multi-state gene network where the gene switches  between its active and inactive state at the jump times \cite{BLPR,tomski15,ZFL}.

Existing models usually describe the underlying phenomena for some population from the point of view of a single individual. In physics this is often known as a particle perspective \cite{bressloff2017stochastic}. That means that the dynamics of every single individual is driven by separate stochastic laws depending on a variety of factors, e.g. its mass or energy. However, there are alternative situations in which the entire population is affected by randomly switching environmental conditions, e.g. particles driven by a common environmental noise \cite{bressloff2017stochastic} or the response of a metabolic or gene regulatory system to an environmental stimulus \cite{belle}. This is known as a population perspective \cite{bressloff2017stochastic},
and is the approach that we use in this paper since we consider one common source of randomness which affects
all individuals in a population.

Here we treat the evolution of the density of a  population distribution in the situation
where every individual has its own deterministic dynamics but  the whole population is affected by some continuous-time Markov process with finite state space that changes the current state of all individuals. In this approach, a state is represented by a population density - an element of an infinite dimensional space.  It is particularly difficult to study the evolution of such densities and
thus  we investigate their moments and correlations of all orders. These infinite dimensional processes are dual to the class of PDMP's known as random evolutions introduced earlier by Griego and Hersh~\cite{griego}, motivated by the work of Goldstein \cite{goldstein} and Kac \cite{Kac}, see \cite{pinsky1991}. (For an amusing and non-technical account  of this history see \cite{hersh1969brownian} and \cite{hersh2003birth}.) They are particular examples of models governed by so-called
 switching Partial Differential Equations (PDEs) that recently have a growing interest in the literature (\cite{bresslofflawley15,lawleyetal2015,lawley2016, berezhkovskii2016diffusive,lawleymurphy}).
 Most of these papers focus on applications in biological sciences. From the mathematical point of view they are based on diffusion processes and PDEs of parabolic type.

In  \cite{bresslofflawley15,lawley} the authors provide the moment and correlation equations in the case of diffusion processes. The current study is a generalization of their work by giving moment and correlation equations for a broader class of processes. The main result of our paper is that the mean of a process described by randomly switching PDEs can be viewed as an appropriate stochastic semigroup (see Theorem~\ref{th1} and Corollary \ref{c:cor2}). This has further important consequences, especially that the mean of random density in the population perspective can be seen as identical to a density from the individual perspective (see Section~\ref{s:examples}).
We study the mean of the process at large time for a variety of examples that are biological applications. It allows us to investigate the asymptotic behaviour for the mean of the process in the cases of fold, transcritical, pitchfork, and Hopf bifurcations.  We also provide numerical simulations for the mean of the process which were prepared by using FiPy~\cite{FiPy:2009}.

This paper is organised as follows. In Section \ref{sec:pre} we provide some basic material from the theory of stochastic semigroups on $L^1$.  Section~\ref{sec:rand-switch-dynam} briefly reviews randomly switching dynamical systems  in Euclidean state spaces. In Section~\ref{s:Dspace} we introduce randomly switching semigroups with the state space being the set of densities leading to a stochastic evolution equation in an $L^1$ space. We study the first moment of its solutions in Section~\ref{s:moment-eqn} where we stress the correspondence between this moment equation  and the Fokker-Planck type equations from Section~\ref{s:Dspace}.   Section~\ref{s:moment-eqn} contains the main results of this paper, namely Theorem \ref{th1} and Corollary \ref{c:cor2}.   The  behaviour of the mean at large time is considered in Section~\ref{s:examples}, where we also give examples of applications of our results to situations in which the underlying dynamics display a variety of bifurcations.
In Section~\ref{sec:hmoment-eqn} we study second order correlations of  solutions of the stochastic evolution equation. We conclude in Section~\ref{sec:concl} with a brief summary. The appendix contains relevant concepts from the theory of tensor products that are used in Section~\ref{sec:hmoment-eqn}.

\section{Preliminaries}\label{sec:pre}

In this section we collect some preliminary material. We begin with  the notion of  stochastic (Markov) semigroups and provide examples of such semigroups.

Let a triple $(E,\E,m)$ be a $\sigma$-finite measure space and let $L^1=L^1(E,\E,m)$. We define  the set of densities $D \subset L^1$ by
\begin{equation*}
D= \{ f \in L^1: f \ge 0, \|f\|=1\}.
\end{equation*}
A \emph{stochastic (Markov) operator} is any linear mapping $P\colon L^1 \rightarrow L^1$ such that $P(D) \subset D$ \cite{almcmbk94}.
A family of linear operators $\{P(t)\}_{t \ge 0}$ on $L^1$ is called a \emph{stochastic semigroup} if each operator $P(t)$ is stochastic and  $\{P(t)\}_{t \ge 0}$ is a \emph{$C_0$-semigroup}, i.e.  the following conditions hold:
\begin{enumerate}
\item $P(0)=\Id,$
\item $P(t+s)=P(t)P(s)$ for all $t,s \ge 0,$
\item for every $f$ the function $t \rightarrow P(t)f$ is continuous.
\end{enumerate}
The infinitesimal \emph{generator} of
$\{P(t)\}_{t\ge0}$ is, by definition, the operator $A$ with domain $\mathcal{D}(A)\subset L^1$ defined as
\begin{align*}
\mathcal{D}(A)&=\{f\in L^1: \lim_{t\downarrow 0}\frac{1}{t}(P(t)f-f) \text{ exists} \},\\
Af&=\lim_{t\downarrow 0}\frac{1}{t}(P(t)f-f),\quad f\in \mathcal{D}(A).
\end{align*}

We will use  stochastic semigroups to represent solutions of evolution equations.
One of the simplest examples of such equations is the deterministic Liouville equation which has a simple interpretation \cite{rudnickipichortyran02}. Consider the movement of particles in the phase space $\mathbb{R}^d$, $d \ge 1$, described
by a differential equation:
\begin{equation}\label{e:intro1}
x'(t)=b(x(t)),
\end{equation}
where $b(x)$ is a $d$-dimensional vector.
Then the Liouville equation describes the evolution of the density of the distribution of particles, i.e.
if $x(t)$ has a density $u(t,x)$, then $u$ is the solution of the following equation:
\begin{equation}\label{e:intro2}
\frac{\partial u(t,x)}{\partial t}=-\dive (b(x)u(t,x)).
\end{equation}
Let,  for any $x_0\in \mathbb{R}^d$, equation \eqref{e:intro1}
with initial condition $x(0)=x_0$ have a solution for all $t$, which we denote by $\pi(t,x_0)$, and let the mapping $x_0\mapsto \pi(t,x_0)$ be \emph{non-singular} with respect to the Lebesgue measure $\Leb$ on $\mathbb{R}^d$, i.e.  $\Leb(\{x\in \mathbb{R}^d:\pi(t,x)\in B\})=0$ for all Borel sets $B$ with $\Leb(B)=0$.
If $f\colon \mathbb{R}^d\rightarrow [0, + \infty)$ is the density of the $\mathbb{R}^d$-valued random vector $\xi _0$, then the density of $\pi(t,\xi _0)$
 is given by
\begin{equation*}
P(t)f(x)=f(\pi(-t,x))\det [\frac{d}{dx}\pi(-t,x)].
\end{equation*}
The family of operators $\{P(t)\}_{t\ge 0}$ forms a stochastic semigroup on the space $L^1(\mathbb{R}^d)$ and $u(t,x)=P(t)f(x)$ is the solution of \eqref{e:intro2} with initial condition
$u(0,x)=f(x) $.

\section{Randomly switching dynamics}\label{sec:rand-switch-dynam}

 In this section we recall a classical setting of PDMP based models seen from the  perspective of individual units and taking place in a finite dimensional space. This well-known situation will be contrasted in the following sections with a population perspective approach, thus moving the analysis to infinite dimensional space.  See \cite[Figure 1]{bressloff2017stochastic} for a nice pictorial distinction between the individual and populational perspectives.

Consider  sufficiently smooth vector fields $b_i$,  $i\in I=\{0,1,\ldots, k\}$, $k \in \mathbb{N}$, defined on an open subset $G$ of $\mathbb{R}^d$. Let $E\subset G$ be a Borel set with non-empty interior and with boundary of Lebesgue measure zero. We assume that for each $i$  and $x_0\in E$ the equation
\begin{equation}\label{e:bi}
x'(t)=b_i(x(t)),
\end{equation}
with initial condition $x(0)=x_0$, has a solution $\pi_i(t,x_0)$ for all $t>0$ in the set $E$. The mapping $(t,x_0)\mapsto \pi_i(t,x_0)$ is continuous. We assume that each $\pi_i(t,\cdot)$ is non-singular with respect to the Lebesgue measure on $E$.
Now let $f\colon E\rightarrow [0, + \infty)$ be a density of an $E$-valued random vector $\xi _0$. Then the density of $\pi_i(t,\xi _0)$ is given by
\begin{equation} \label{d:Pit}
P_{i}(t)f(x)=\mathbf{1}_{E}(\pi_{i}({-t},x))f(\pi_{i}({-t},x))\det [\frac{d}{dx}\pi_{i}({-t},x)].
\end{equation}
For every $i \in I$ the family of operators $\{P_{i}(t)\}_{t\ge 0}$ forms a stochastic semigroup on $L^1(E)$ called a Frobenius-Perron semigroup \cite[Section 7.4]{almcmbk94}.

A randomly switching dynamics is a Markov process $\xi(t)=(x(t),i(t))$ on the state space $E\times I$ such that the dynamics of $x(t)$ is given by the solution of the equation
\begin{equation}\label{d:xit}
 x'(t)=b_{i(t)}\big(x(t)\big),
\end{equation}
and $i(t)$ is a continuous-time Markov chain with values in $I$ and intensity matrix $[q_{ij}]$. If the system initially is at time $t_0$  at the state $(x_0,i)$ then $x(t)$  changes in time according to equation \eqref{e:bi} as long as $i(t)=i$ and $\xi(t)=(\pi_i(t-t_0,x_0),i)$. If $i(t)$ changes its value to some $j$  with intensity $q_{ij}$, i.e.  the probability of switching from $i$ to $j$ after time $\Delta t$ is $q_{ij}\Delta t+o(\Delta t)$, $j\in I, j\neq i$, then we choose the vector field $b_j$ and we start afresh. Let $t_1$ be the moment of switching from $i$ to $j$  and $x_1=\pi_i(t_1-t_0,x_0)$. Then  we have  $\xi(t)=(\pi_j(t-t_1,x_1 ),j)$ until the next change of the state of the process $\{i(t)\}_{t\ge 0}$. This construction repeats indefinitely.
We set
\begin{equation}\label{d:qi}
  q_i=\sum_{j\neq i}q_{ij}\quad \text{and}\quad  q_{ii}=-q_{i}, \quad i\in I.
\end{equation}
Note that if $k=1$, then we have $q_{10}=-q_{11}=q_1$ and $q_{01}=-q_{00}=q_0$.

For each $i\in I$, let $\{P_i(t)\}_{t \ge 0}$
be the stochastic semigroup as in \eqref{d:Pit} and let  $(A_i,\mathcal{D}(A_i))$ be its generator. If $f=(f_i)_{i\in I}$ is a column vector
consisting of functions $f_i$ such that $f_i\in \mathcal D(A_i)$, we set
$Af=(A_if_i)_{i\in I}$ which is also a column vector.
We  denote the matrix $[q_{ij}]$ by $Q$ and its transpose $[q_{ji}]$ by $Q^T$.
Then the operator $A+Q^T$ is the infinitesimal generator of a stochastic semigroup $\{P(t)\}_{t\ge 0}$ on the space $L^1(E\times I)=L^1(E\times I,\mathcal B(E\times I),\mu)$. Here  $\mathcal B(E\times I)$ is the $\sigma$-algebra of Borel subsets of  $E\times I$ and  $\mu$ is the product of the $d$-dimensional Lebesgue measure and the counting measure on $I$.

Thus if we define $u(t)=P(t)f$,  then $u$ satisfies the   evolution equation
\begin{equation}\label{e:FPE1}
\begin{cases}
u'(t)=A u(t) + Q^T u(t), \\ u(0)=f.
\end{cases}
\end{equation}
Now using the notation $f=(f_i)_{i \in I}$ and $u(t)=\left(u_i(t)\right)_{i \in I}$, we can rewrite \eqref{e:FPE1} as
\begin{equation}\label{e:FPE}
\begin{cases}
u_i'(t)=A_i u_i(t) + \sum_{j\in I} q_{ji}u_j(t),\\ u_i(0)=f_i, \quad i \in I.
\end{cases}
\end{equation}
Moreover,
the process
$\xi(t)=(x(t),i(t))$ induces the stochastic semigroup $\{P(t)\}_{t\ge 0}$ (see \cite[Section 4.2]{rudnickityran17}), i.e. if $f$
is the density of $\xi (0)$ then $P(t)f$ is the density of $\xi (t)$
and
\begin{equation*}
\Pr(x(t) \in B, i (t)=i)=\int_{B\times\{i\}} P(t)f d\mu=\int_B u_{i}(t,x) \, dx, \quad i \in I.
 \end{equation*}

\section{Randomly switching densities}\label{s:Dspace}

In this section we look at the role of stochasticity in explaining biological phenomena from the point of view of the {\it whole population}
in an environment affected by some random disturbances.  We illustrate this approach using two examples, namely a population model with two different birth rates \cite{rudnickityran15}, and a model of inducible gene expression with positive feedback on gene transcription \cite{grif}.

\begin{example}[Population model with two different birth rates]\label{ex:trans}
Consider a population of size $x \ge 0$, a death rate $\mu$ and birth rate $\beta-cx$ with $\beta$ changing in time between two possible values $\beta_0$ and $\beta_1$ in response to some environmental disturbance. The growth of the population is assumed to be determined by the differential equation \eqref{e:bi} with
\begin{equation}\label{e:R111}
b_0(x) =(\beta_0-cx)x - \mu x \quad \text{and} \quad b_1(x)
=(\beta_1-cx)x- \mu x.
\end{equation}
For each value of $i$, equation \eqref{e:bi} induces a stochastic semigroup $\{P_i(t)\}_{t \ge 0}$ given by equation \eqref{d:Pit}. If the population initially grows with birth rate related to a value $i\in\{0, 1\}$ and has a distribution density $ g$, then after time $t>0$ the population distribution density is given by $P_i(t) g$. We consider the switching between the semigroups $\{P_0(t)\}_{t \ge 0}$ and $\{P_1(t)\}_{t \ge 0}$ according to a Markov chain $i(t) \in \{0,1\}$ and obtain a stochastic Liouville equation
(as introduced by Bressloff in \cite{bressloff2017stochastic}):
\begin{equation}\label{ex1.liouville}
  \frac{\partial u(t,x)}{\partial t}=-\frac{\partial \big(b_{i(t)}(x)u(t,x)\big)}{\partial x},
\end{equation}
where $u(t,x)=P_i(t) g(x)$ for $i(t)=i$ is the population density.
This is an infinite dimensional version of equation \eqref{d:xit}.
\end{example}

\begin{example}[One dimensional inducible Goodwin model with positive feedback on gene transcription]\label{ex:fold}
We consider the inducible operon model \cite{grif} which describes the expression of genes driven by positive feedback control of gene transcription.  This provides an effective mechanism by which a protein can maintain the expression of its own gene as well as switch between two levels of expression (un-induced and induced). We look at a cluster of identical copies of a selected gene, e.g. the cluster of multiple copies of the same gene in the case of bacteria, where one of these genes is 16S ribosomal RNA - the component of prokaryotic ribosome \cite{ribo}. We denote the concentration level at time $t$ by $x(t) \ge 0$
and assume the degradation rate is changing with constant intensities between two values $\gamma_0\text{ and }\gamma_1$  driven by  environmental noise.

For a given $i\in\{0,1\}$ the concentration level is assumed to evolve according to the nonlinear differential equation:
\begin{equation}\label{e:Fi}
x'(t)=\frac{x^n(t)}{1+x^n(t)}-\gamma_i x(t),
\end{equation}
where $n$ is a natural number.  This equation induces a semigroup $\{P_i(t)\}_{t \ge 0}$ as in formula \eqref{d:Pit}. If the gene cluster initially has a degradation rate $\gamma_i$ and a distribution density $ g$ then after time $t>0$ the population distribution density is given by
$P_i(t) g$. By taking $b_i(x)$ equal to the right-hand side of equation \eqref{e:Fi} we obtain the stochastic Liouville equation \eqref{ex1.liouville} as in the previous example.
\end{example}

We next consider the general case of a population of individuals living in an environment with random disturbances.
This situation with a population of particles affected by a common environmental noise was considered from a statistical physics viewpoint by Bressloff in \cite{bressloff2017stochastic}.
The dynamics of a population is described by equation \eqref{e:bi} with a state space being some Borel subset of $\mathbb{R}^d$
and $i$ taking values from a given finite set $I$. The environmental disturbances correspond to switching between  different values of $i$. Thus the total population density $u(t,x)$ is described by the equation:
\begin{equation}\label{d:SLE}
\frac{\partial u(t,x)}{\partial t}=-\dive(b_{i(t)}(x) u(t,x)),
\end{equation}
where $i(t)$ is a Markov chain on a discrete state space $I$. We supplement equation~\eqref{d:SLE} with an
initial condition $u(0,x)=g(x)$. Note that equation \eqref{d:SLE} is a more general case of the stochastic Liouville equation \eqref{ex1.liouville}. Randomly switching environments have also been analyzed in the case of diffusion processes \cite{lawley}.
In such situations the stochastic Liouville equation is replaced by an appropriate parabolic equation. We may generalize all these
cases by looking at a general scheme for randomly switching stochastic semigroups.

Let $L^1(E)=L^1(E,\B(E),m)$, where $(E,\rho)$ is a separable metric space and $m$ is a  $\sigma$-finite measure,  and let $I=\{0,1,\ldots,k\}$, $k \in \mathbb{N}$. For every
$i \in I$, consider a linear operator $(A_i,\mathcal{D}(A_i))$ which is the generator of a stochastic semigroup $\{P_i(t)\}_{t\ge 0}$
on $L^1(E)$. We assume that the stochastic process $\{i(t)\}_{t\ge 0}$ is a continuous time Markov chain with state space $I$ and constant intensities $q_{ij}$, $i,j \in I$. We consider the following stochastic evolution equation:
\begin{equation}\label{d:pdm}
 \begin{cases}
 u'(t)=A_{i(t)}u(t),\\
 u(0)=g.
 \end{cases}
\end{equation}
Equation \eqref{d:pdm} generates a well-defined PDMP~\cite[p. 32]{rudnickityran17}
\begin{equation}\label{def:process}
X(t)=(u(t),i(t)),\quad t\ge 0,
\end{equation}
with values in the space $L^1(E)\times I$.

Now we derive the solution $u(t)$ of \eqref{d:pdm}.
Let $t_0=0$ and for each $n\in \mathbb{N}$  let $t_n$ be the $n$th jump time of the Markov process $i(t)$ so that
\[
\Pr(t_{n+1}-t_n>t|i(t_n)=i)=e^{-q_i t}, \quad  n\ge 0.
\]
Let $\mathbb{P}_j$ be the probability measure defined on sample paths $\omega$ of the process $\{i(t)\}_{t\ge 0}$ with $i(0)=j$. Denote integration with respect to the measure  $\mathbb{P}_j$ by $\mathbb{E}_j$,
set $U(t_0)=U(0)=\Id$ and define
\begin{equation*}
U(t)= P_{i(t_n)}(t-t_n)\circ U(t_n)\quad \text{for } t\in[t_n,t_{n+1}), n\ge 0,
\end{equation*}
where
\[
U(t_n)=P_{i(t_{n-1})}(t_{n}-t_{n-1})\circ \dots\circ P_{i(t_1)}(t_2-t_1)\circ P_{i(0)}(t_1).
\]
Further, let
\begin{equation}\label{d:Nt}
N(t)=\max \{n \ge 0:  t_n\le t \}
\end{equation}
be the number of jumps of the process $\{i(t)\}_{t\ge 0}$ up to time $t$. Using $N(t)$ we can write
\begin{equation}\label{d:Ut}
U(t)=P_{i(t_{N(t)})}(t-t_{N(t)})U(t_{N(t)}),\quad t\ge 0.
\end{equation}
Then $u(t)=U(t) g$ for $g \in L^1(E)$ is the solution of \eqref{d:pdm}. Note that $U(t)$ depends on $\omega$. For each $t>0$, $g\in L^1(E)$, and $\omega$, if $t\in [t_n(\omega),t_{n+1}(\omega))$ for some $n=n(\omega)$ then for  $i=i(t_n(\omega))$ we have
\[
U(t,\omega) g=P_i(t-t_n(\omega))U(t_n(\omega)) g,
\]
and  $U(t,\omega)$ is a composition of a finite number of stochastic operators. Thus
\[
\|U(t,\omega) g\|=\int_E |U(t,\omega) g(x)|m(dx)\le  \int_E |g(x)|m(dx)=\|g\|
\]
and  $U(t) g$ can be regarded as a random variable with values in $L^1(E)$. It is, in fact,  a Bochner integrable $L^1(E)$-valued random variable, by \cite[Theorem 3.7.4]{hille}. Furthermore, for each $j\in I$, $\mathbb{E}_j(U(t)g)$ is represented by the a.e. finite function $\mathbb{E}_{j}(u(t,x))$  where $u(t,x)=U(t)g(x)$, i.e. for any $h\in L^{\infty}(E, \mathcal{B}(E),m)$
\begin{equation*}
\int_E \mathbb{E}_j(U(t)g)(x)h(x) m(dx) =\int_E \mathbb{E}_j(u(t,x))h(x)m(dx).
\end{equation*}

In the next sections we derive equations for the moments of $u(t,x)$ and compare them to the evolution
equation \eqref{e:FPE1} from Section \ref{sec:rand-switch-dynam}.

\section{First moment equations}\label{s:moment-eqn}

We continue with the general setting from Section \ref{s:Dspace} and study the first
moment of the solution $u(t)$ of \eqref{d:pdm}.  Recall that $u(t,x)=U(t)g(x)$,
where $U(t)$ is given by \eqref{d:Ut}.
Using a simple decomposition, the first moment of $u(t,x)$ can be written  as
\begin{equation*}
V(t,x)=\mathbb{E}(u(t,x))= \sum_{j\in I}\mathbb{E}_j(u(t,x))= \sum_{j\in I} \sum_{i\in I}\mathbb{E}_j (\mathbf{1}_{\{i(t)=i\}}u(t,x)), \quad x\in E,  t\ge 0,
\end{equation*}
where $\mathbb{E}$ denotes the expectation and $\mathbf{1}_F$ is the indicator function of an event $F$.
If we take
\begin{equation}\label{d:Vi}
V_i(t,x)=\sum_{j\in I} \mathbb{E}_j (\mathbf{1}_{\{i(t)=i\}}u(t,x)),
\end{equation}
then
\begin{equation}\label{d:V}
V(t,x)=\sum_{i\in I} V_i(t,x), \quad x\in E,  t\ge 0.
\end{equation}
We will show that:
\begin{equation}\label{l:lawley}
\frac{\partial}{\partial t} V_i = A_i V_i + \sum_{j} q_{ji}V_j,\quad i\in I.
\end{equation}
It will turn out that $(V_i)_{i\in I}$ can be represented as a stochastic semigroup.

Let $L^1(E\times I)=L^1(E\times I,\mathcal{B}(E\times I), \mu )$, where $\mu$ is the product of
the measure $m$ on $E$ and the counting measure on $I$.  We denote elements of the space $L^1(E\times I)$ by
$ f :=(f_i)_{i \in I} \in L^1(E\times I)$, where $f_i \in L^1(E)$, $i \in I$, and define the family of operators $\{P(t)\}_{t \ge 0}$ on the space $L^1(E\times I)$ by
\begin{equation}\label{theorem1a}
(P(t)f)_i=\sum_{j \in I} \mathbb{E}_j (\mathbf{1}_{\{i(t)=i\}}U(t)f_{j})
, \quad i\in I.
\end{equation}
We impose the following conditions:
\begin{enumerate}[(I)]
\item\label{a:1} There exist a Banach space $B$, a set of Borel measurable functions $\mathcal{H}\subseteq B$, and a  family $\mathcal{C}$ of Borel subsets of $E$ that is closed under intersections and generates the Borel $\sigma$-algebra $\mathcal{B}(E)$    such that  for each set $F\in \mathcal{C}$ there is a nonincreasing sequence of function $h_n\in \mathcal{H}$ satisfying
    \begin{equation}\label{e:approx}
    \lim_{n\to\infty}h_n(x)=\mathbf{1}_F(x),\quad   x\in E.
    \end{equation}
\item\label{a:2} Let $\{P_i(t)\}_{t\ge 0}$ be a stochastic semigroup on $L^1(E)$ with generator $(A_i,\mathcal{D}(A_i))$, $i\in I$.  For each $i\in I$  there is a $C_0$-semigroup $\{T_i(t)\}_{t\ge 0}$ on $B$ such that
\begin{equation}\label{a:adjoint}
\langle P_i(t)g,h\rangle=\langle g, T_i(t)h\rangle, \quad g\in L^1(E), h\in B, t\ge 0.
\end{equation}
Here, the scalar product of two functions $g,h$ with their domain
 $E$ is defined by $\langle g, h\rangle :=\int_E g(x) h(x) m(dx)$.
\end{enumerate}

Now we state and prove  one of  the main results of this paper.

\begin{theorem}\label{th1} Assume conditions \eqref{a:1} and \eqref{a:2}.
Then the family of operators $\{P(t)\}_{t \ge 0}$ defined in \eqref{theorem1a} is a stochastic semigroup on $L^1(E\times I)$.
Moreover, the infinitesimal generator of the semigroup $\{P(t)\}_{t\ge 0}$
is the operator $A+Q^T$
where
$Af =(A_i f_i)_{i \in I}$ and
$Q^Tf =(\sum_{j} q_{ji}f_j)_{i \in I}$
for $f =(f_i)_{i \in I}$ with  $f_i\in \mathcal{D}(A_i)$, $i\in I$.
\end{theorem}

\begin{proof}
  Given  $B$ and $\{T_i(t)\}_{t\ge 0}$, $i\in I$,  as in conditions \eqref{a:1} and \eqref{a:2}, we define   the random evolution family $\{ M(t), t \ge 0\}$ of operators  on the space $B $ by  \cite{griego}
\begin{equation*}
M(t)= T_{i(0)}(t_1) T_{i(t_1)}(t_2-t_1)\cdots T_{i(t_{N(t)})}(t-t_{N(t)}),
\end{equation*}
where $N(t)$ is as in \eqref{d:Nt}. Consider the product space $B \times B \times \cdots \times B = B^{k+1}$, where $k+1$ denotes the number of elements of the set $I$, and for any $  h :=(h_i)_{i \in I} \in  B^{k+1}$,
$i \in I$ define
\begin{equation}\label{d:g1}
( T (t)  h  )_i=\mathbb{E}_i (M(t) h_{i(t)}).
\end{equation}
Griego and Hersh  \cite{griego} showed that the family
$\{ T (t)\}_{t \ge 0 }$ forms a
strongly continuous semigroup of bounded linear operators on $B^{k+1}$.
The integral in \eqref{d:g1} with respect to $\mathbb{P}_i$ is understood in the sense of Bochner.
Observe that $M(t)$ depends on $\omega$ and that $\omega\mapsto M(t,\omega) h_{i(t)(\omega)}$ is $\mathbb{P}_j$-Bochner integrable for each $j\in I$ (see \cite[Lemma 2]{griego}).
By Hille's lemma for the Bochner integral \cite[Theorem 3.7.12]{hille}  in $B$ we obtain
\begin{equation}\label{Hille}
\langle \mathbb{E}_j (M(t) h_{i(t)}),g \rangle =\mathbb{E}_j (\langle M(t)h_{i(t)} ,g  \rangle), \quad j\in I, g\in L^1(E), h\in  B^{k+1},
\end{equation}
where the integral on the right-hand side of \eqref{Hille} is in the sense of Lebesgue.

For $f \in L^1(E\times I)$ and $ h \in  B^{k+1}$,  set
\[
\langle f , h\rangle=\sum_{j\in I}\langle f_j,h_j\rangle=  \sum_{j\in I}\int_E  f_j(x)h_j(x)m(dx).
\]
We first show that
\begin{equation}\label{d:g2}
\langle  T (t)  h,f \rangle = \langle  h,P(t) f  \rangle.
\end{equation}
It follows from \eqref{Hille} that
\[
\langle  T (t)  h,f  \rangle =\sum_{j \in I} \langle \mathbb{E}_j (M(t) h_{i(t)}),f_j \rangle
= \sum_{j \in I} \mathbb{E}_j (\langle M(t)h_{i(t)} ,f_{j} \rangle).
\]
Using \eqref{a:adjoint} it is easily seen that  \[
\langle M(t)h_{i(t)} ,f_{j} \rangle=\langle h_{i(t)} ,U(t)f_{j} \rangle, \quad t\ge 0.
\]
Hence
\begin{align*}
\langle  T (t)  h,f \rangle =
\sum_{j \in I} \mathbb{E}_j (\langle h_{i(t)}, U(t) f_{j} \rangle)=\sum_{j,i \in I} \mathbb{E}_j (\langle h_{i} , \mathbf{1}_{\{i(t)=i\}}U(t) f_{j} \rangle).
\end{align*}
Using Hille's lemma for Bochner integrals in $L^1(E)$, we see that
\[
\langle  T (t)  h,f \rangle =\sum_{i \in I} \langle h_{i},\sum_{j \in I} \mathbb{E}_j ( \mathbf{1}_{\{i(t)=i\}}U(t) f_{j} )\rangle=\langle h,P(t) f  \rangle,
\]
as claimed.

Now, we check the semigroup property $P(t+s)f =P(t)\circ P(s)f $ for $t,s\ge 0$, $f \in L^1(E\times I)$.
Since $\{ T (t)\}_{t\ge 0}$ is a semigroup, it follows from \eqref{d:g2} that
\begin{equation*}
\langle  T (t) \circ  T (s)   h,f  \rangle=\langle  T (t+s)  h,f  \rangle =\langle  h, P(t+s) f  \rangle
\end{equation*}
for any $t,s \ge 0$ and $f \in L^1(E\times I),  h\in  B^{k+1}$. This together with \eqref{d:g2} gives
\begin{equation*}
\langle  h, P(t+s) f  \rangle = \langle  T (s)  h,  P(t) f  \rangle = \langle   h,P(s) \circ P(t) f  \rangle,
\end{equation*}
implying that for each $i\in I$ and $h\in \mathcal{H}$ we have
\begin{equation}\label{e:semi}
\int_E {h} (P(t+s) f )_i dm=\int_E{h}(P(s) \circ P(t) f )_idm.
\end{equation}
Since for each $i$ the semigroup $\{P_i(t)\}_{t\ge 0}$ is stochastic, we see that each operator $P(t)$ is stochastic  on $L^1(E\times I)$. Thus decomposing an arbitrary  $ f $ into its positive and negative parts, we can assume that $f \in L^1(E\times I)$ is nonnegative.
Since \eqref{e:semi} holds for each $h_n$, the Lebesgue convergence theorem implies that \eqref{e:semi} holds for $\mathbf{1}_F$, showing that
\begin{equation}\label{e:semis}
\int_{F} (P(t+s) f )_idm=\int_F (P(t)\circ P(s)f )_idm
\end{equation}
for all sets $F\in \mathcal{C}$.
The family $\mathcal{C}$ is a $\pi$-system, i.e., $F_1\cap F_2\in \mathcal{C}$ for $F_1,F_2\in \mathcal{C}$, and equality \eqref{e:semis} holds for all $F\in \mathcal{C}\cup\{E\}$.  Hence we conclude that \eqref{e:semis} holds for all Borel subsets of $E$. Consequently, $(P(t+s) f )_i=(P(t)\circ P(s)f )_i$ for all $t,s\ge 0$ and $i\in I$.
Since  almost all sample paths of the stochastic switching process $i(t)$ are right-continuous functions, we conclude that $\{P(t)\}_{t\ge 0}$
is a $C_0$-semigroup, completing the proof that $\{P(t)\}_{t\ge 0}$ is a stochastic semigroup.

Finally, it was shown in \cite{griego} that the generator $\mathcal{L}$ of the semigroup $\{ T (t)\}_{t\ge 0}$ is given by
\[
({\mathcal{L}} h)_i=\mathcal{L}_ih_i+\sum_{j\in I} q_{ij}h_j
\]
for $ h=(h_i)_{i\in I}\in  B^{k+1}$ with $h_i\in \mathcal{D}(\mathcal{L}_i)$, $i\in I$, where $(\mathcal{L}_i,\mathcal{D}(\mathcal{L}_i))$ is the generator of the semigroup $\{T_i(t)\}_{t\ge 0}$ on $B$, $i\in I$.
Observe that
\[
\langle f ,{\mathcal{L}} h \rangle =\sum_{i\in I}
\langle f_i,\mathcal{L}_ih_i+\sum_{j\in I} q_{ij}h_j \rangle =\sum_{i\in I}
\langle f_i,\mathcal{L}_ih_i\rangle +\sum_{i\in I}
\sum_{j\in I}\langle f_i, q_{ij}h_j \rangle
\]
for $f =(f_i)_{i\in I}$ with $f_i\in \mathcal{D}(A_i)$, $i\in I$.
Since we have
\[
\langle f_i,\mathcal{L}_ih_i\rangle=\langle A_if_i,h_i\rangle, \quad i\in I,
\]
by assumption \eqref{a:2}, we conclude that
\[
\langle f ,{\mathcal{L}} h \rangle=\sum_{i\in I}
\langle A_if_i,h_i\rangle +
\sum_{j\in I}\langle \sum_{i\in I}q_{ij}f_i,h_j \rangle=\langle (A+Q^T)f , h \rangle,
\]
implying the form of the generator of the semigroup $\{P(t)\}_{t\ge 0}$.
\end{proof}

We next show that Theorem~\ref{th1} can be applied to
semigroups of Frobenius-Perron operators given by \eqref{d:Pit}.

\begin{corollary}\label{c:1}
Let $E$ be a Borel subset of $\mathbb{R}^d$ and for each $i \in I$ let  $\{P_i(t)\}_{t\ge 0}$ be given by \eqref{d:Pit}. Then conditions \eqref{a:1} and \eqref{a:2} hold, and  Theorem~\ref{th1} does also.
\end{corollary}
\begin{proof}
Let $B$ be the space $C(E)$ of continuous functions on $E$  if $E$ is compact or the space $C_0(E)$ of continuous functions on $E$ which vanish at infinity, otherwise. Recall that the $\sigma$-algebra
of Borel subsets of $E$ is generated by the family of compact sets. For each compact set $F$  the function
$h_n$ defined
by
\begin{equation*}
h_n(x)=\max\{1-n\rho(x,F),0\},
\end{equation*}
where $\rho(x,F)$ denotes the distance of the point $x$ from the set $F$, is globally Lipschitz. Since $h_n$ belongs to $B$ and satisfies \eqref{e:approx}, we conclude that condition \eqref{a:1} holds.
We define $T_i(t)\colon B\to B$  by
$T_i(t) h(x)=h(\pi_i(t,x)),$ $t\ge 0$, $x\in E$, $h \in B$. Note that $T_i(t)$ is a $C_0$-semigroup on $B$, see \cite[Section B-II]{nagel86-0} and condition \eqref{a:adjoint} holds, see \cite[Section 7.4]{almcmbk94}.
\end{proof}

Using Theorem \ref{th1} we obtain the following,  which is the second main result of this paper: \begin{corollary}\label{c:cor2} Assume conditions \eqref{a:1} and \eqref{a:2}. Let $\{P(t)\}_{t\ge 0}$ be given by \eqref{theorem1a} and $u(t)$ by~\eqref{d:pdm}. For each $g\in L^1(E)$ and $l\in I$ such that $u(0)=g$ and $i(0)=l$ we have $V_i(t,x)=(P(t)f )_i(x),$
where $V_i$ is as in \eqref{d:Vi} and $f =(f_j)_{j \in I}$ is of the form
\[
f_j=
\begin{cases}
g,\quad j=l, \\ 0,\quad j \neq l.
\end{cases}
\]
In particular, the mean of the process in  \eqref{def:process} is given by
\begin{equation}\label{e:meanP}
V(t,x)=\mathbb{E}(u(t,x))=\sum_{i\in I}(P(t)f )_i(x).
\end{equation}
\end{corollary}

\section{Mean of the process at large time}\label{s:examples}

We consider the relationship between Fokker-Planck type systems \eqref{e:FPE1}
 for a distribution of processes in $\mathbb{R}^d$ space, and
the first moment equation \eqref{l:lawley}. The latter has the same form as \eqref{e:FPE} and $P(t)f $ is the solution of
the evolution equation \eqref{e:FPE1} with initial condition $f $.
Hence, if $(u_0,u_1,u_2,...,u_k)$ is a solution of \eqref{e:FPE}
and for each $i \in I$ there exists $f_i^* \in L^1(E)$ such that
\begin{equation}\label{e:con1}
 \lim_{t \rightarrow \infty}  \int_{E}\left| u_i(t,x)- f_i^*(x) \right| m(dx)=0,
\end{equation}
then, by Corollary~\ref{c:cor2}, we have
\begin{equation*}
\lim_{t \rightarrow \infty} \int_{E} \left|V_i(t,x)- f_i^*(x) \right| m(dx)=0
\end{equation*}
and, by equation \eqref{d:V},
\begin{equation}\label{e:con33}
\lim_{t \rightarrow \infty} \int_{E} \left| V(t,x)- V^*(x) \right| m(dx)=0
\end{equation}
where
$
  V^*(x)=\sum_{i \in I} f_i^*(x).
$
 The function $V^*$ is called the \emph{mean of the process at large time}.
In particular, condition \eqref{e:con1} holds if the semigroup $\{P(t)\}_{t\ge 0}$  is \emph{asymptotically stable}, i.e. there exists $f^* \in L^1(E\times I)$
such that for each density $f\in L^1(E\times I) $
\begin{equation}\label{e:ast}
\lim_{t \to \infty} \| P(t)f  - f^* \| =0
\end{equation}
Note that  $f^*$  in \eqref{e:ast} is an \emph{invariant density} for $\{P(t)\}_{t\ge 0}$, i.e. $P(t)f^*=f^*$ for all $t\ge 0$.

On the other hand, if the semigroup $\{P(t)\}_{t\ge 0}$ is \emph{sweeping from compact subsets of $E\times I$}, i.e. for each compact subset $F$ of $E$, any $f\in L^1(E\times I)$ and $i\in I$ we have
\[
\lim_{t\to \infty}\int_{F} (P(t)f)_i(x)m(dx)=0,
\]
then  the mean of the process in \eqref{def:process} at large time is equal to zero, since for any compact subset $F$ of $E$ we have
\begin{equation}
\lim_{t\to \infty}\int_{F} V(t,x)m(dx)=0.
\end{equation}

We now provide sufficient conditions  for asymptotic behaviour of the stochastic semigroup $\{P(t)\}_{t\ge 0}$ induced by the randomly switching dynamics $\xi(t)=(x(t),i(t))$ with $x(t)$ satisfying \eqref{d:xit} as described in Section~\ref{sec:rand-switch-dynam} with $E\subseteq \mathbb{R}^d$. We follow  the work of \cite{bakhtinhurth12,benaim12} and \cite{rudnickityran15,rudnickityran17}.

Recall that the Lie bracket of two sufficiently smooth vector fields $b_i$ and $b_j$ is defined by
\[
[b_i,b_j](x)=Db_j(x)b_i(x)-Db_i(x)b_j(x)
\]
where $Db(x)$ is the derivative of the vector field $b$ at point $x$. Given vector fields $b_0,\ldots,b_k$  sufficiently smooth in a neighbourhood of $x$ we say that  \emph{H\"ormander's condition holds at} $x$ if the vectors
 \[
b_1(x)-b_0(x), \dots, b_k(x)-b_0(x),\,
[b_i,b_j](x)_{0\leq i,j\leq k},\,
[b_i,[b_j,b_l]](x)_{0\leq i,j,l\leq k},\dots
\]
span the space $\mathbb{R}^d$. This condition is called
the hypo-ellipticity condition A in \cite{bakhtinhurth12} and the strong bracket condition in \cite{benaim12}.

From \cite[Theorem 2]{bakhtinhurth12} (see also \cite[Theorem 4.4]{benaim12}) and \cite[Corollary 5.3]{rudnickityran17} we obtain the following
\begin{corollary}
Suppose that H\"ormander's condition holds at every $x\in E$. If the semigroup $\{P(t)\}_{t\ge 0}$ has no invariant density, then it is sweeping from compact subsets of $E\times I$.
\end{corollary}

A point $x\in E$ is called \emph{reachable from} $y$ if we can find  $n\in \mathbb{N}$, indices $i_1,\ldots,i_n\in I$ and times $s_1,\ldots s_n$ such that $x=\pi_{i_n}(s_n,\ldots,\pi_{i_1}(s_1,y))$, where for each $i$ the function  $t\mapsto \pi_i(t,x_0)$ is the solution of \eqref{e:bi} with initial condition $x(0)=x_0$. Finally, a point $x$ is called \emph{accessible from} $y$ if each neighborhood  of $x$ contains a point reachable from~$y$. Now, combining  \cite[Theorem 1]{bakhtinhurth12} with \cite[Threorem 4.6]{benaim12} we have

\begin{corollary}\label{c:coras}
Suppose that the semigroup $\{P(t)\}_{t\ge 0}$ has an invariant density.  If   H\"ormander's condition holds at a point $x\in E$ that is accessible from any point in $E$ then the semigroup $\{P(t)\}_{t\ge 0}$ is asymptotically stable.
\end{corollary}

We illustrate the behaviour of the mean of the process at large time
 for some simple examples exhibiting bifurcations in their trajectory dynamics \cite{kuznetsov2013elements}.  Note that in \cite{lawley18} direct bifurcations of the heat equation with randomly switching boundary conditions were studied.   We will use the results from
\cite{hurth} and \cite{rudnickityran15}, and we start by recalling some notions from them.

We consider $E\subset (0,\infty)$  and $I=\{0,1\}$ so that we have a switching between $b_0$ and $b_1$ leading to a Markov process $\xi (t)=(x(t),i(t))$, $t \ge 0$ on the state space $E \times \{0,1\}$.
By $q_i=q_{i(1-i)}$, $i=0,1$, (see \eqref{d:qi}) we denote constant positive intensities of switching  from $b_i$ to $b_{1-i}$.
Additionally we assume that $b_0(0)=b_1(0)=0$ and that either $b_0$ or $b_1$ has  one more stationary point $a$  that is accessible from any point in  $E$.  H\"ormander's condition holds at $x$ if $b_1(x)-b_0(x)\neq 0$.
Let
\begin{equation*}
r(x)=\frac{q_{0}}{b_0(x)}+\frac{q_{1}}{b_1(x)}\quad  \text{and}\quad R(x)= \int_{x_0}^x r(s) ds,
\end{equation*}
where $x_0 \in (0,a)$. Then the functions given by
\begin{equation*}
f_0(x)=\frac{e^{-R(x)}}{|b_0(x)|}\quad  \text{and}\quad  f_1(x)=\frac{e^{-R(x)}}{|b_1(x)|}
\end{equation*}
are stationary solutions of the corresponding Fokker-Planck equation \eqref{e:FPE}. Now if
\begin{equation}\label{eq:denn3}
\kappa=\int_{0}^a \big(f_0(x)+f_1(x)\big) dx < \infty,
\end{equation}
then the semigroup $\{P(t)\}_{t\ge 0}$ is asymptotically stable and the mean at large time is given by
\begin{equation}\label{e:Vst}
  V^*(x)=\kappa^{-1}( f_0(x) +f_1(x))\mathbf{1}_{(0,a)}(x).
\end{equation}
If $b_0'(0)b_1'(0)\neq 0$ then
 condition \eqref{eq:denn3}
holds for $\lambda>0$ where this parameter depends on the form of the functions $b_0$, $b_1$ and is defined by
\begin{equation}\label{d:lambda}
\lambda=p_0b'_0(0)+p_1 b'_1(0)
\end{equation}
with
\[
p_0=\frac{q_{1}}{q_{0}+q_{1}} \quad \text{and}\quad p_1=\frac{q_{0}}{q_{0}+q_{1}}
\]
representing the probability of choosing the function $b_0$ and $b_1$, respectively.
In the opposite situation with $\lambda<0$ this semigroup is sweeping from the family of all compact subsets of the state space implying that the mean at large time is zero.
The parameter $\lambda$ turns out to be the mean growth rate if the population is small \cite{rudnickityran15}.

\begin{example}[Transcritical bifurcation]\label{ex:transcritical} Transcritical bifurcations appear in many biological models, c.f \cite{boudjellaba2009dynamic,kribs,buonomo2015note,wang2014hopf}, and we thus  re-consider Example \ref{ex:trans}. The functions $b_0$ and $b_1$ are given by \eqref{e:R111} with $\beta_0<\mu$ and $\beta_1>\mu.$ Thus, \eqref{e:bi}
with  $i=0$ has the form $x'=(\beta_0-cx)x - \mu x$ and there are two stationary points, $0$ and $a_0=(\beta_0-\mu)/c$, where the first one is stable and the second is unstable.
However, for $i=1$ the quantitative character of the stationary points of   $x'=(\beta_1-cx)x - \mu x$ is exchanged. That is, $0$ is an unstable stationary point while $a_1=(\beta_1-\mu)/c$ is stable. Hence, we have a transcritical bifurcation. We take $a=a_1$.
Again, we look at the value of the parameter $\lambda$ in \eqref{d:lambda}.
For $\lambda<0$ the mean of the process at large time is $0$, while for $\lambda>0$ the mean is positive and given by the corresponding $V^*$
with $f_i$, $i=0,1$ given, up to a multiplicative constant, by
 \begin{equation}\label{eq:den3}
f_i(x)=  \dfrac{1}{x|x-a_i|} x^{-(\frac{q_{0}}{ca_0} + \frac{q_{1}}{ca_1})}
(x-a_0)^{\frac{q_{0}}{ca_0}}
(a_1-x)^{\frac{q_{1}}{ca_1}}\mathbf{1}_{(0,a_1)}(x).
\end{equation}
  We illustrate the behavior of the mean $V(t,x)$ as in \eqref{e:meanP} for chosen times and  parameters in  Figures \ref{fig1} and \ref{fig2a}.  Figure \ref{fig1} shows convergence of $V$ to the mean at large time $V^*$ for $\lambda>0$ while Figure \ref{fig2a} presents sweeping to 0  for $\lambda<0$.
\begin{figure}[tb]
\centering
\begin{subfigure}[b]{\figwidth}
\includegraphics[width=\figwidth]{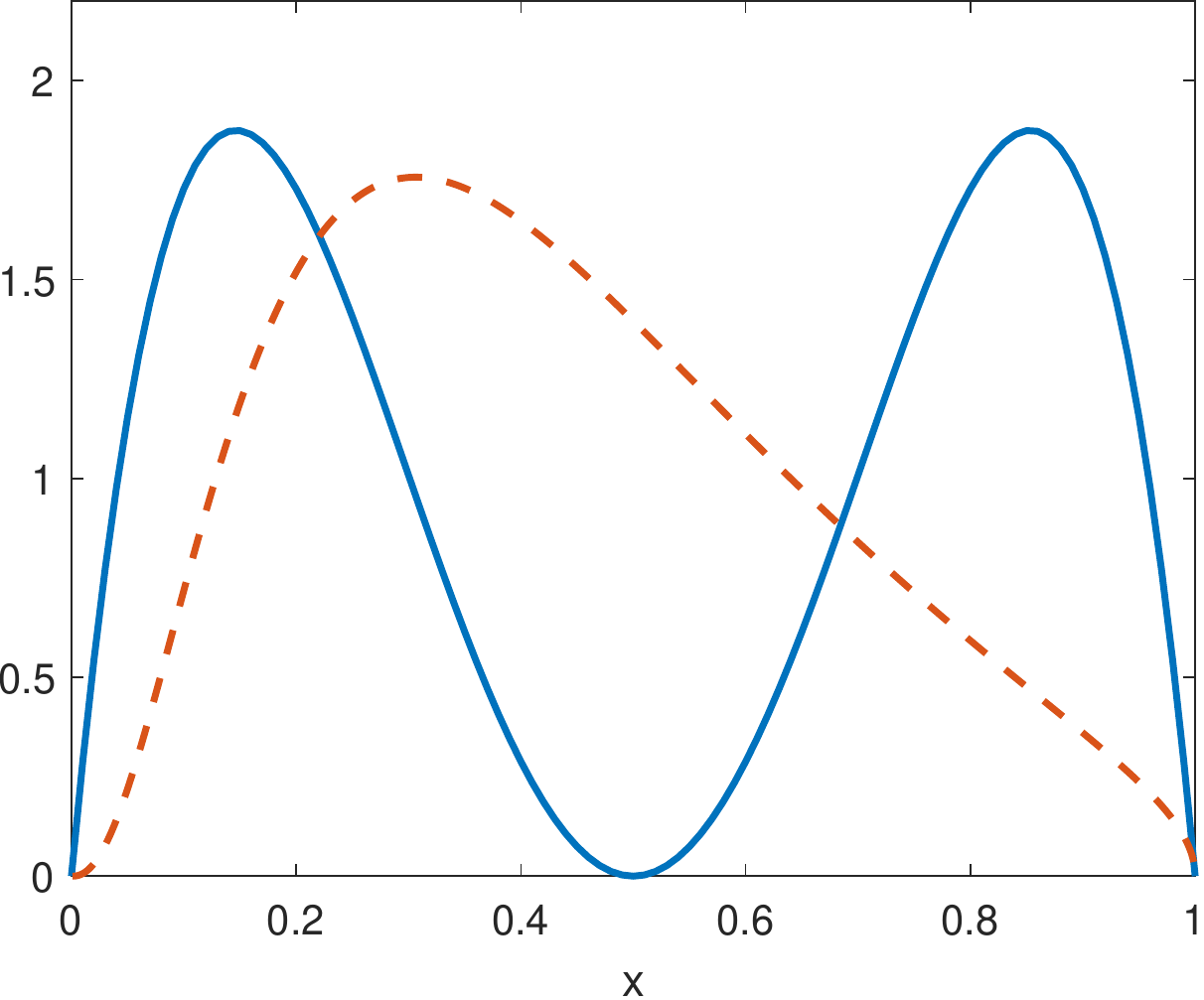}
\caption{}
\end{subfigure}
\begin{subfigure}[b]{\figwidth}
\includegraphics[width=\figwidth]{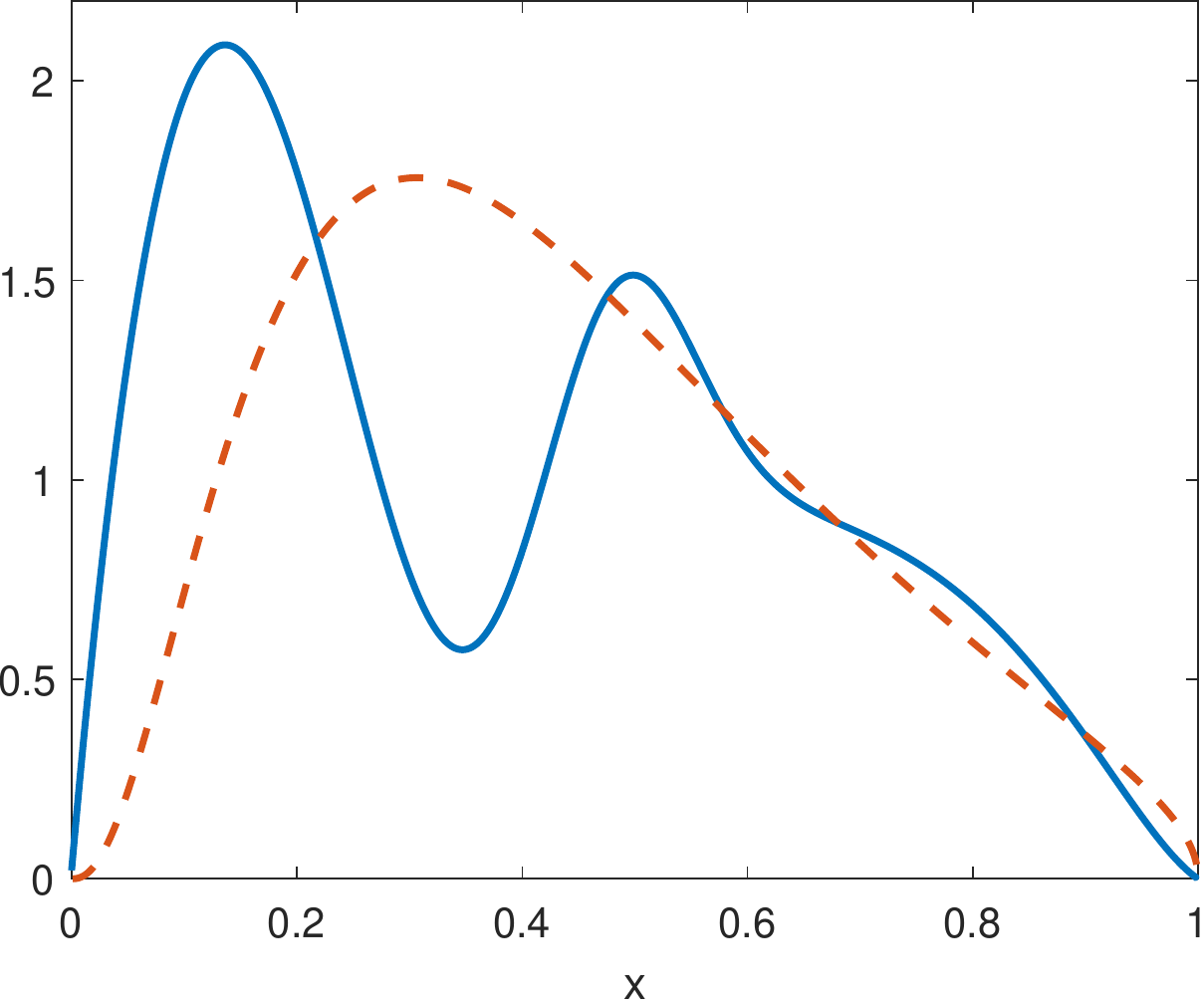}
\caption{}
\end{subfigure}
\begin{subfigure}[b]{\figwidth}
\includegraphics[width=\figwidth]{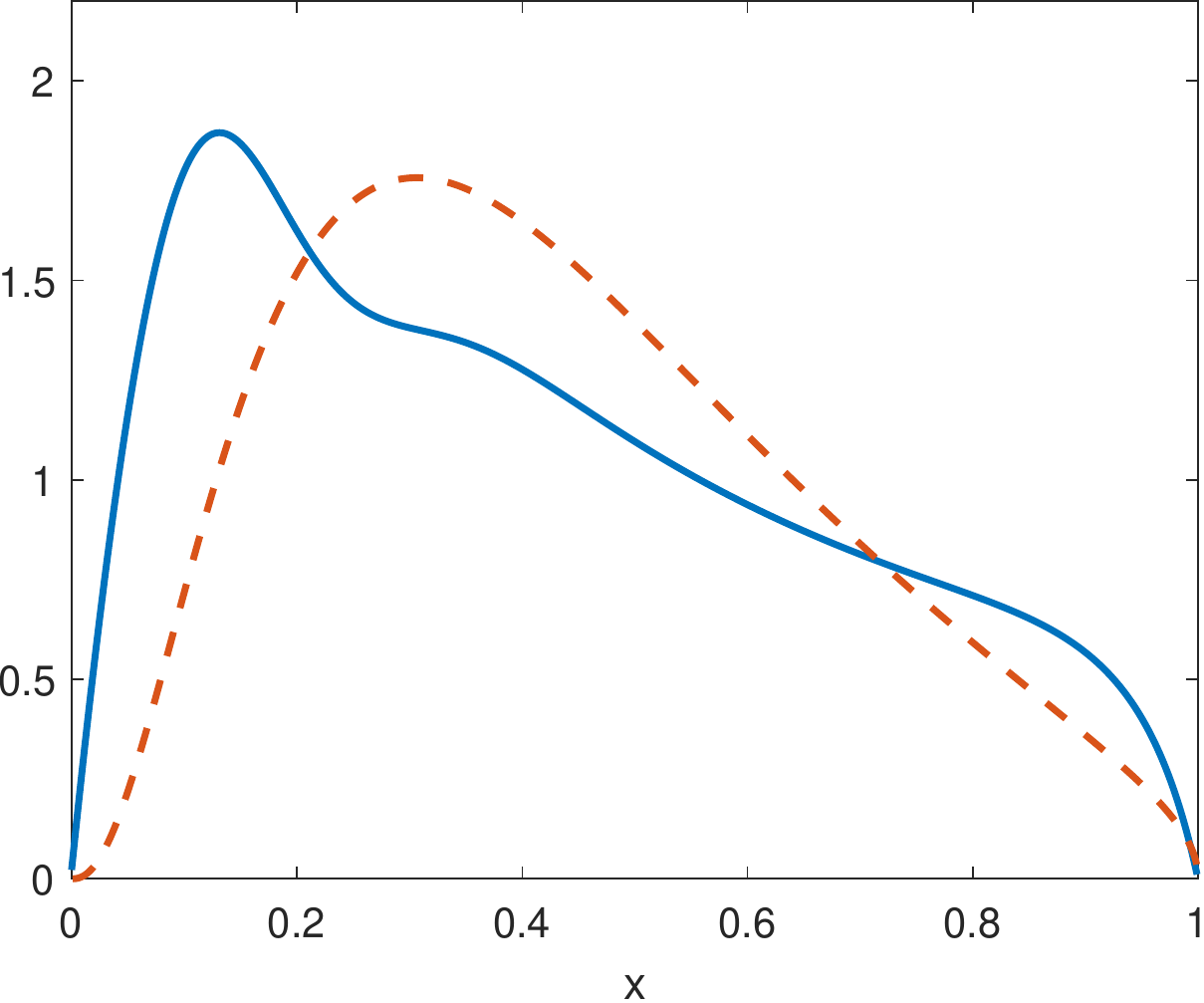}
\caption{}
\end{subfigure}
\begin{subfigure}[b]{\figwidth}
\includegraphics[width=\figwidth]{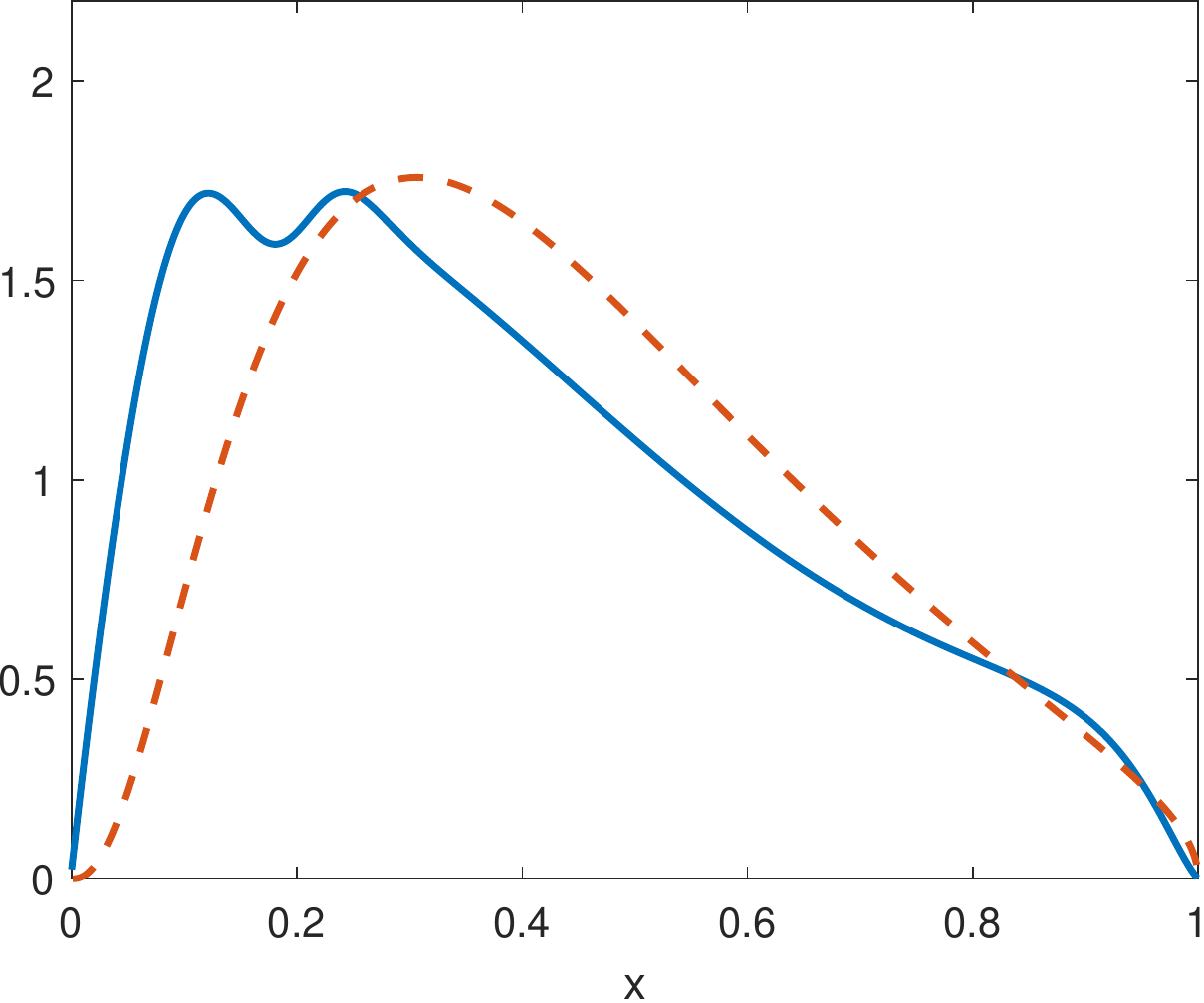}
\caption{}
\end{subfigure}
\begin{subfigure}[b]{\figwidth}
\includegraphics[width=\figwidth]{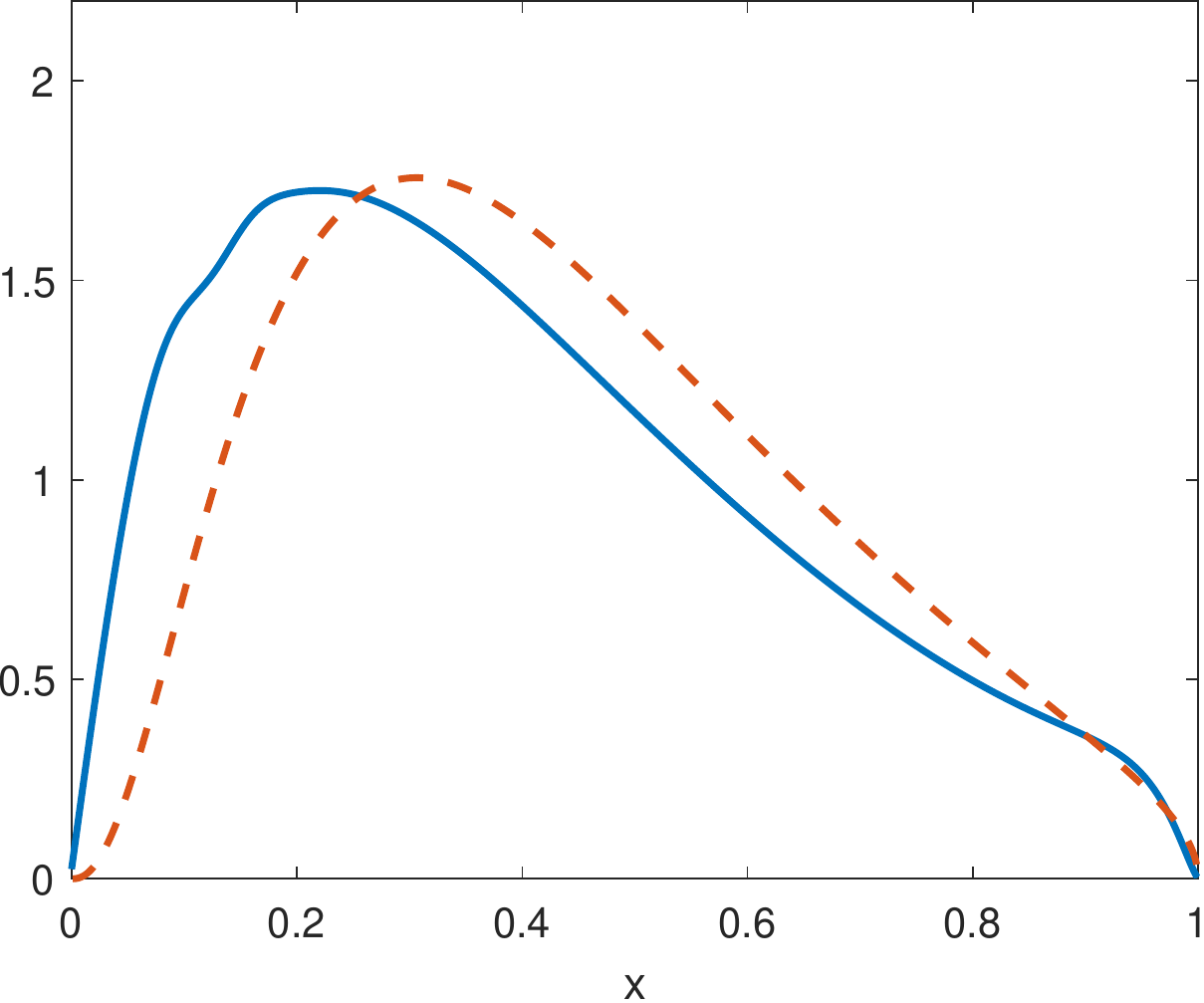}
\caption{}
\end{subfigure}
\begin{subfigure}[b]{\figwidth}
\includegraphics[width=\figwidth]{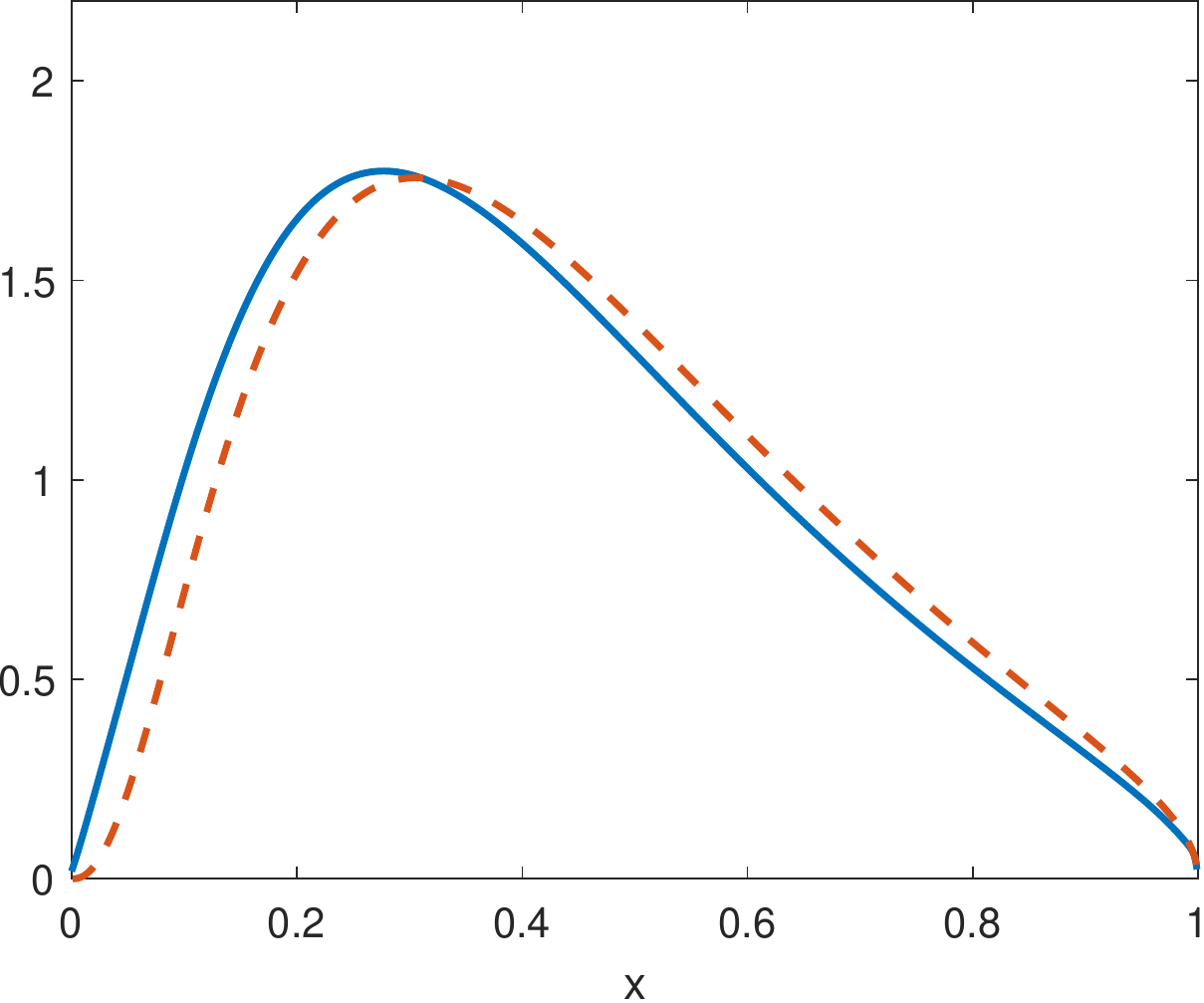}
\caption{}
\end{subfigure}
\caption{Numerical simulations of the mean $V(t,x)$ (solid line) in Example~\ref{ex:transcritical}. The initial density is shown in (a). The dashed line represents the graph of the mean in large time $V^*(x)$. Consecutive times are $t=0.25$ (b), $t=0.5$ (c), $t= 0.7$ (d), $t= 1$ (e), and $t= 2.5$ (f). The values of parameter used in this example are: $q_0=5$, $q_1=3$, $\beta_0=1$, $\beta_1=4$, $c=2$, $\mu=2$\label{fig1}
}
\end{figure}

\end{example}

\begin{figure}[tb]
\centering
\begin{subfigure}[b]{\figwidth}
\includegraphics[width=\figwidth]{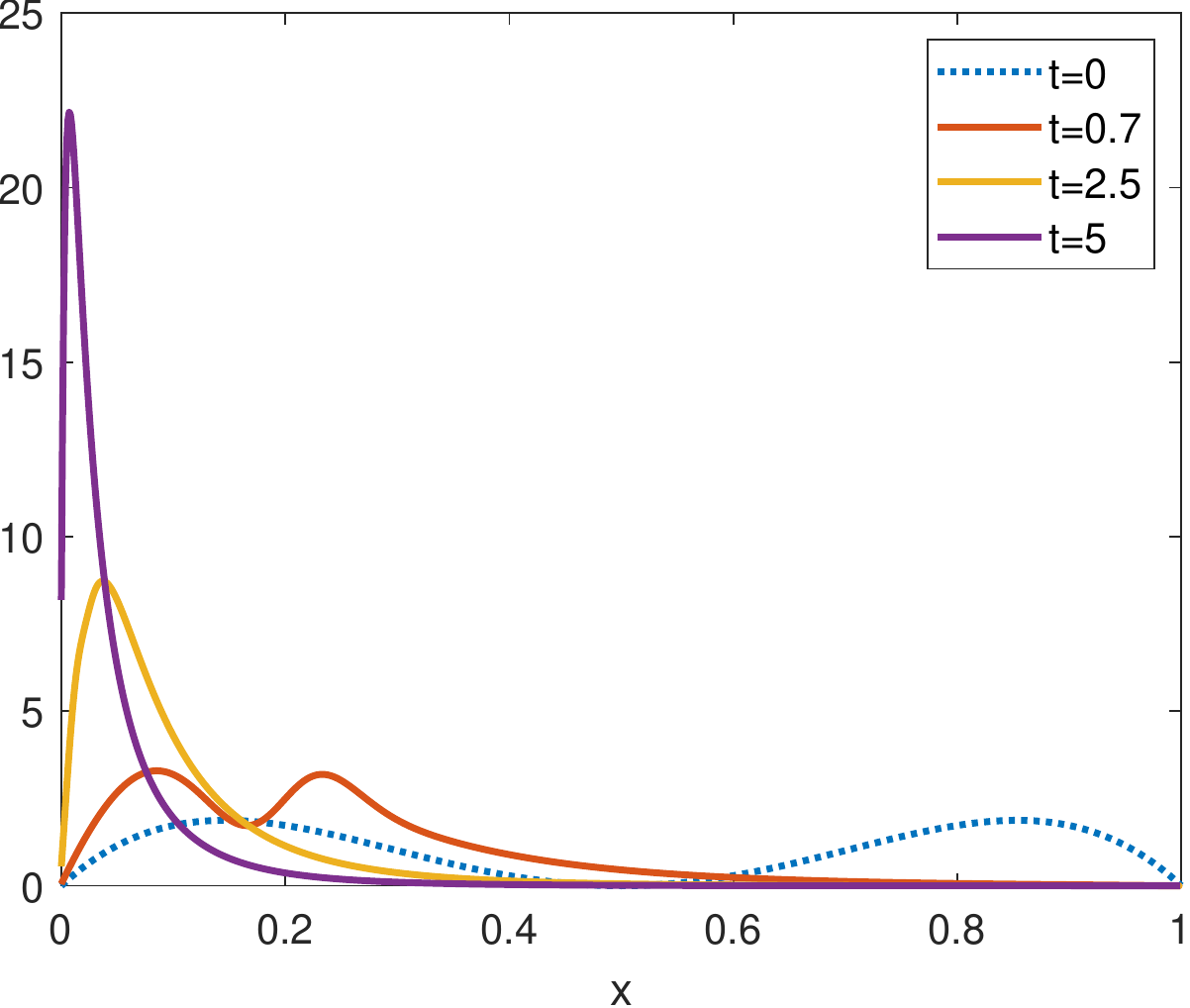}
\caption{\label{fig2a}}
\end{subfigure}
\begin{subfigure}[b]{\figwidth}
\includegraphics[width=\figwidth]{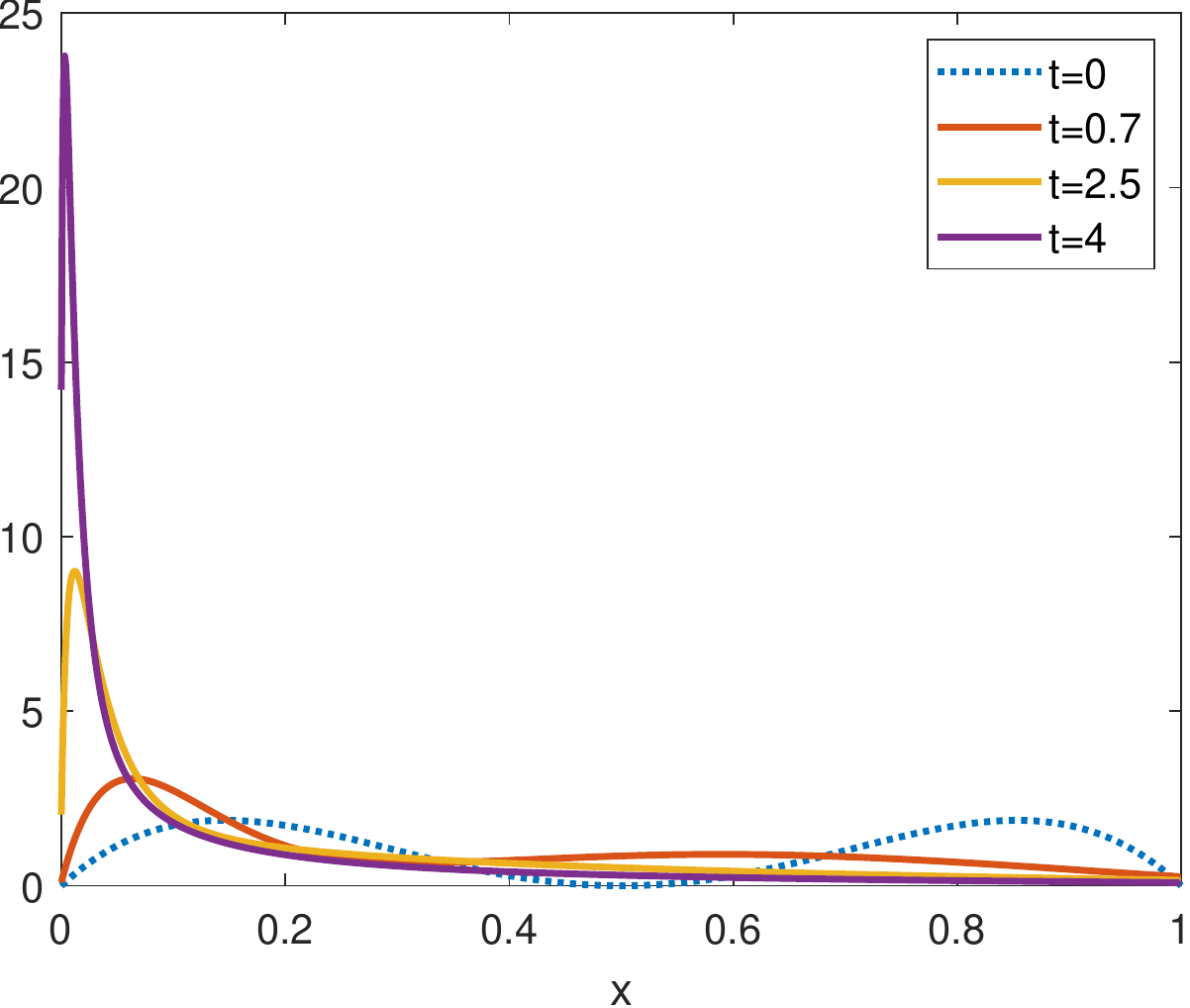}
\caption{\label{fig2b}}
\end{subfigure}
\begin{subfigure}[b]{\figwidth}
\includegraphics[width=\figwidth]{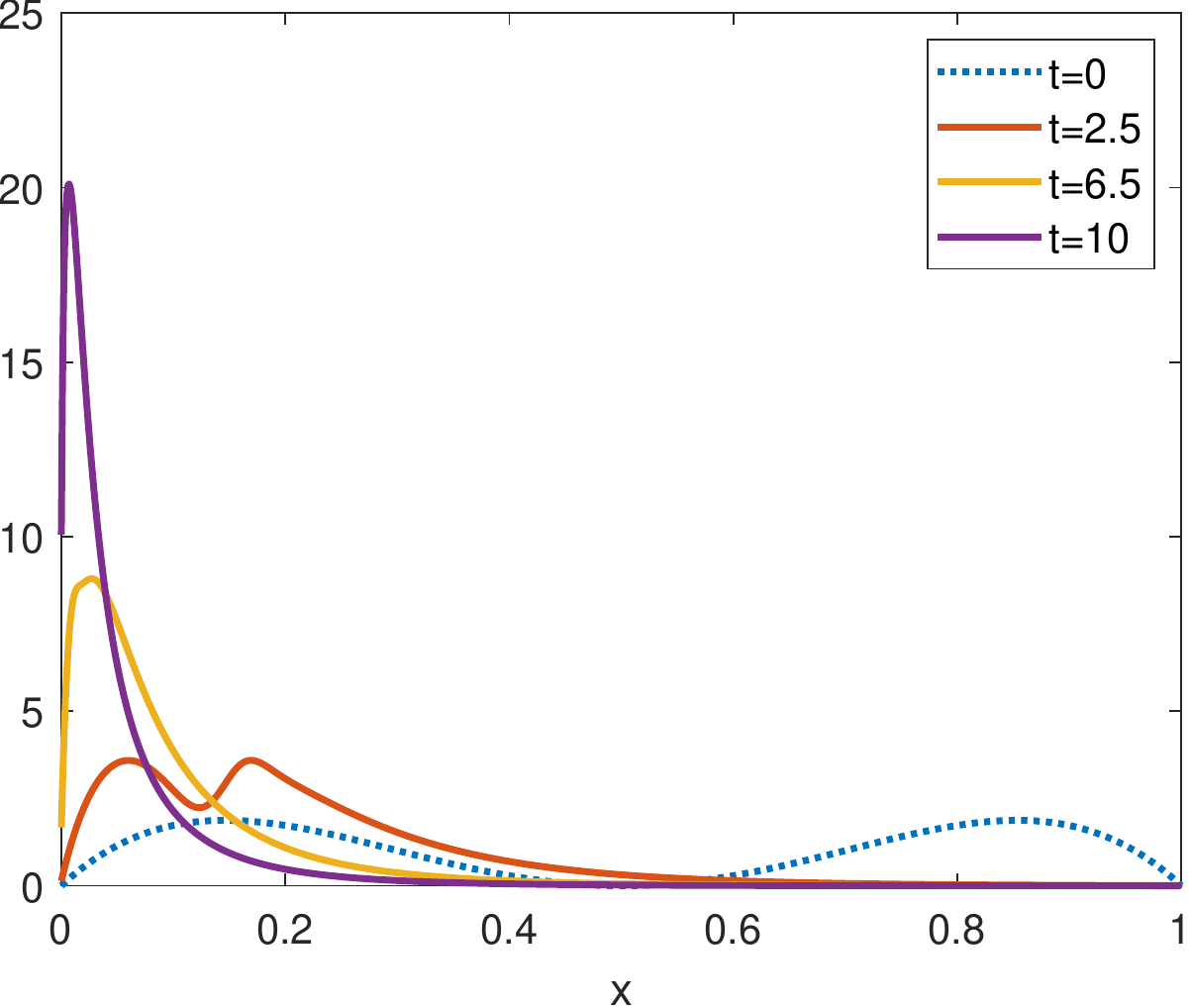}
\caption{ \label{fig2c}}
\end{subfigure}
\caption{Numerical simulations of $V(t,x)$ when the mean at large time is $0$: (a) Example~\ref{ex:transcritical} with parameters:  $q_0=2$, $q_1=6$, $\beta_0=1$, $\beta_1=4$, $c=2$, $\mu=2$, (b) Example~\ref{ex:fold} with parameters:  $q_0=6$, $q_1=2$,$\gamma_0=2$, $\gamma_1=0.25$, (c) Example~\ref{ex:pitch} with parameters:  $q_0=4$, $q_1=2$, $\alpha_0=-0.5$, $\alpha_1=1$ \label{fig2}}
\end{figure}

\begin{example}[Fold bifurcation]  We next go back to the inducible operon model of Example \ref{ex:fold}. Consider the following nonlinear differential equation
\begin{equation}\label{e:F}
x'(t)=\frac{x^n(t)}{1+x^n(t)}-\gamma x(t),
\end{equation}
where $x(t)>0$ denotes the concentration level of protein molecules at time $t$, $\gamma$ is a degradation rate and $n>1$.
It is known (see  \cite{grif}) that if the parameters satisfy the condition
\begin{equation}\label{e:G} n^n \gamma^n > {(n-1)}^{n-1},
\end{equation}
then $0$ is the only stationary point of equation \eqref{e:F} and it is stable. In the opposite case to \eqref{e:G}, there are also two additional stationary points of this equation; one of them is stable and the other one is unstable. Hence, we choose the values of the parameters $\gamma_0$ and $\gamma_1$ in such way that a fold bifurcation occurs. Thus we take $\gamma_0$ such that  $n^n \gamma_0^n > {(n-1)}^{n-1}$ and $\gamma_1$ such that $n^n \gamma_1^n < {(n-1)}^{n-1}$. By using the same type of argument as in the proof of \cite[Theorem 4.2]{hurth} together with the properties of the dynamics $x(t)$, we see that $x(t)$ reaches a neighborhood
of $0$ in finite time and hence we deduce that the semigroup $\{P(t)\}_{t\ge 0}$ is sweeping from the family of all compact subsets of  $(0,+\infty) \times \{0,1\}.$ Consequently, the mean at large time is equal to zero.   This behavior is illustrated in  Figure \ref{fig2b}.
\end{example}

\begin{example}[Pitchfork bifurcation]\label{ex:pitch}
The normal form for a  supercritical pitchfork bifurcation is
\begin{equation}\label{pitch}
x'(t)=\alpha x(t)-x^3(t).
\end{equation}
For $\alpha< 0$ there is a single stationary point $x_*=0$ while for $\alpha > 0$ there is an unstable stationary point at $0$ and two stationary points $x_{\pm}= \pm \sqrt{\alpha}$ that are locally stable. Now let $\alpha_0<0$ and $\alpha_1>0$ be two fixed parameters. We consider equation \eqref{pitch} with $\alpha=\alpha_i$. Thus we have $b_0(x)=\alpha_0 x-x^3$ and $b_1(x)=\alpha_1 x-x^3$.

First we take $E=(0,\infty )$ and $a=\sqrt{\alpha_1}$. Then for positive $\lambda$ from \eqref{d:lambda} we have a positive mean at large time with $f_i$ given by
\begin{equation}\label{pitch2}
f_i(x)=\dfrac{1}{x|\alpha_i-x^2|} x^{-\frac{q_{0}}{\alpha_0}-\frac{q_{1}}{\alpha_1}}(x^2-\alpha_0)^{\frac{q_{0}}{2\alpha_0}}(\alpha_1-x^2)^{\frac{q_{1}}{2\alpha_1}}1_{(0,\sqrt{\alpha_1})}(x),\ i \in \{0,1\},
\end{equation}
while for $\lambda<0$ the mean of the process at large time is equal to $0$. The situation with $E=(-\infty,0)$ is analogous to stationary solutions of the corresponding Fokker-Planck equation given by $f_i(-x),\ x<0$ where this function is as in \eqref{pitch2}.  The behavior of the mean $V(t,x)$ in this example is shown in  Figures \ref{fig3} and \ref{fig2c}.  The convergence of $V$ to the mean at large time $V^*$ for $\lambda>0$ is illustrated in Figure \ref{fig3} while sweeping to 0  for $\lambda<0$ is pres1ented in Figure \ref{fig2c}.

\begin{figure}[tb]
\centering
\begin{subfigure}[b]{\figwidth}
\includegraphics[width=\figwidth]{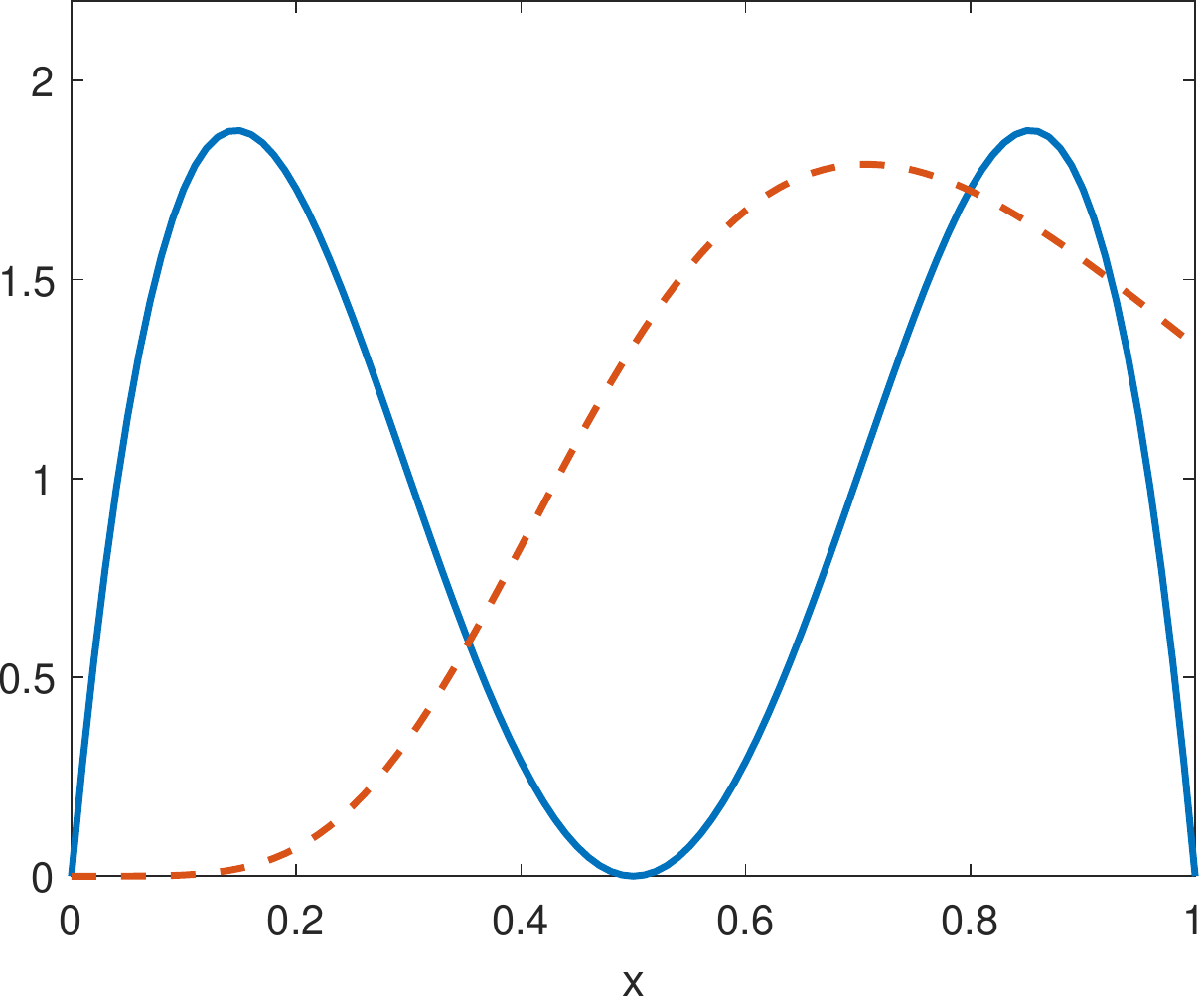}
\caption{}
\end{subfigure}
\begin{subfigure}[b]{\figwidth}
\includegraphics[width=\figwidth]{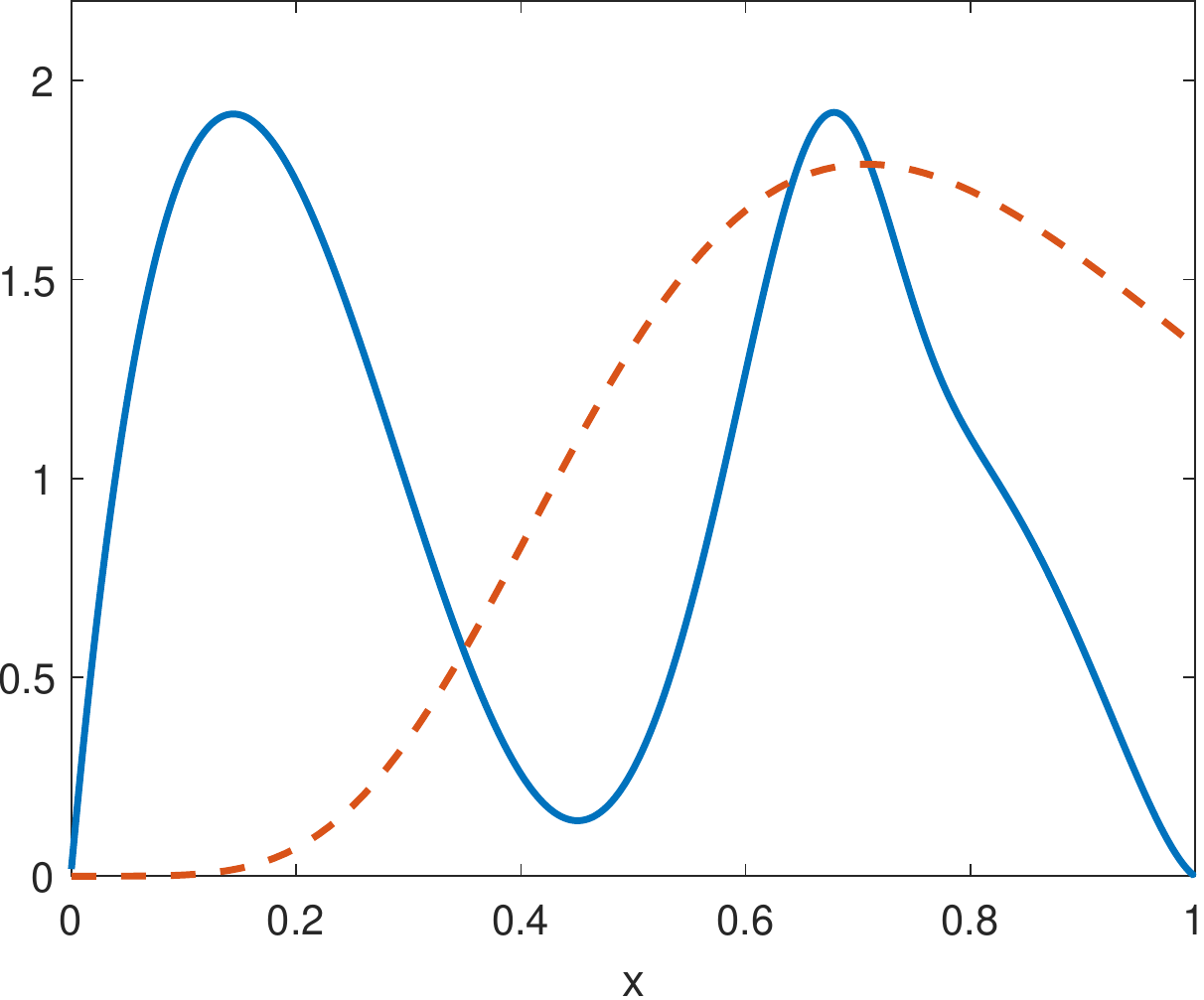}
\caption{}
\end{subfigure}
\begin{subfigure}[b]{\figwidth}
\includegraphics[width=\figwidth]{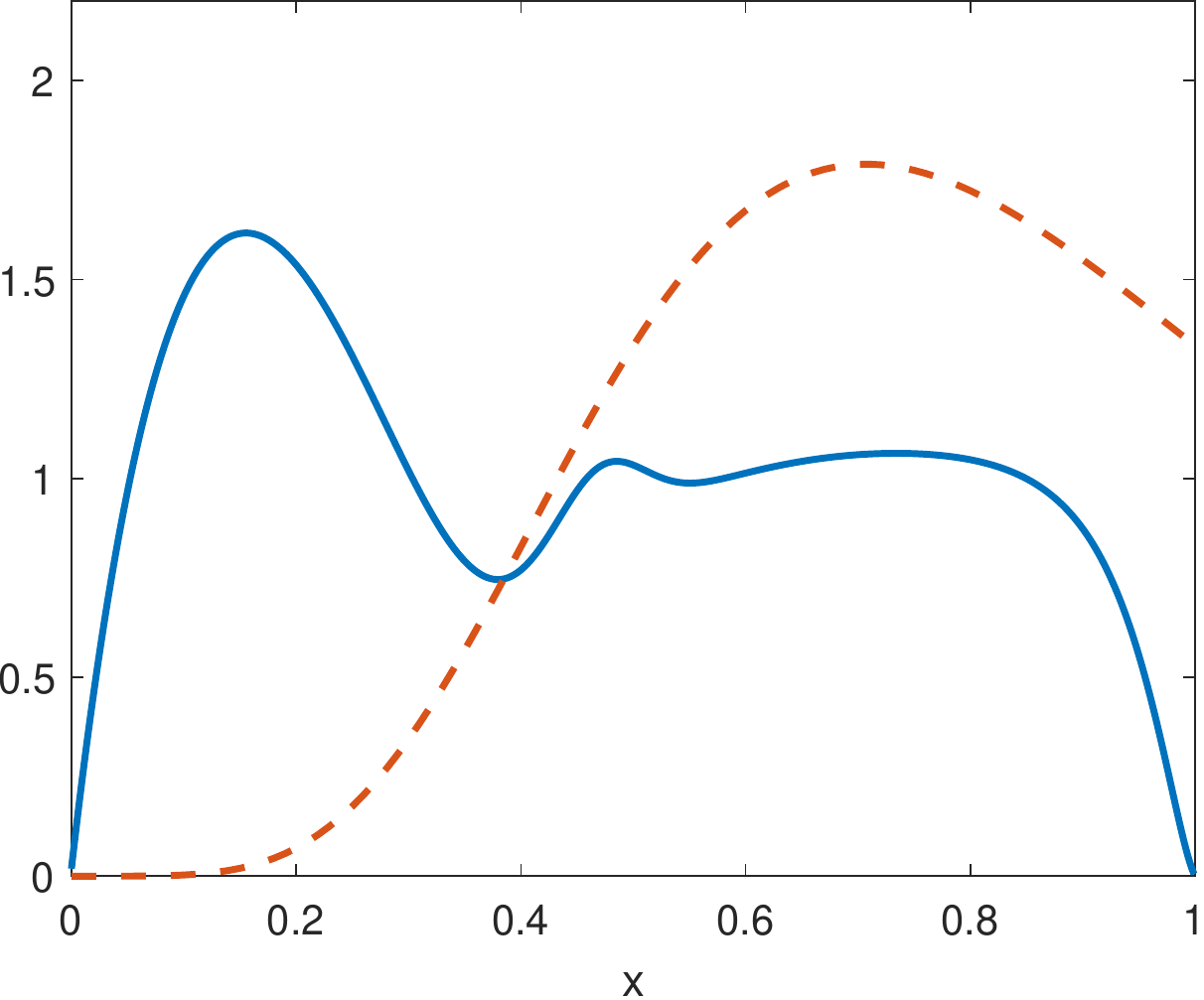}
\caption{}
\end{subfigure}
\begin{subfigure}[b]{\figwidth}
\includegraphics[width=\figwidth]{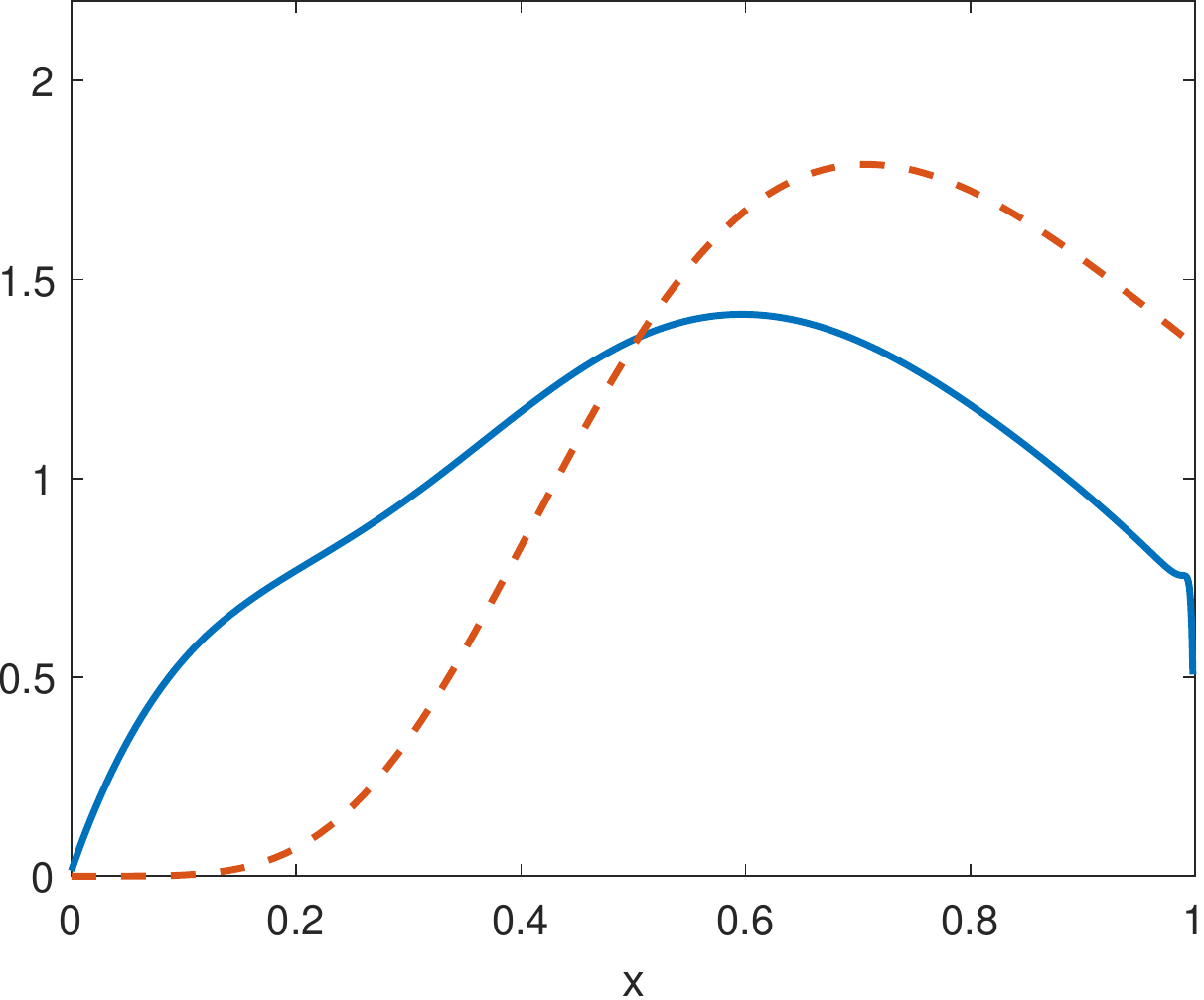}
\caption{}
\end{subfigure}
\begin{subfigure}[b]{\figwidth}
\includegraphics[width=\figwidth]{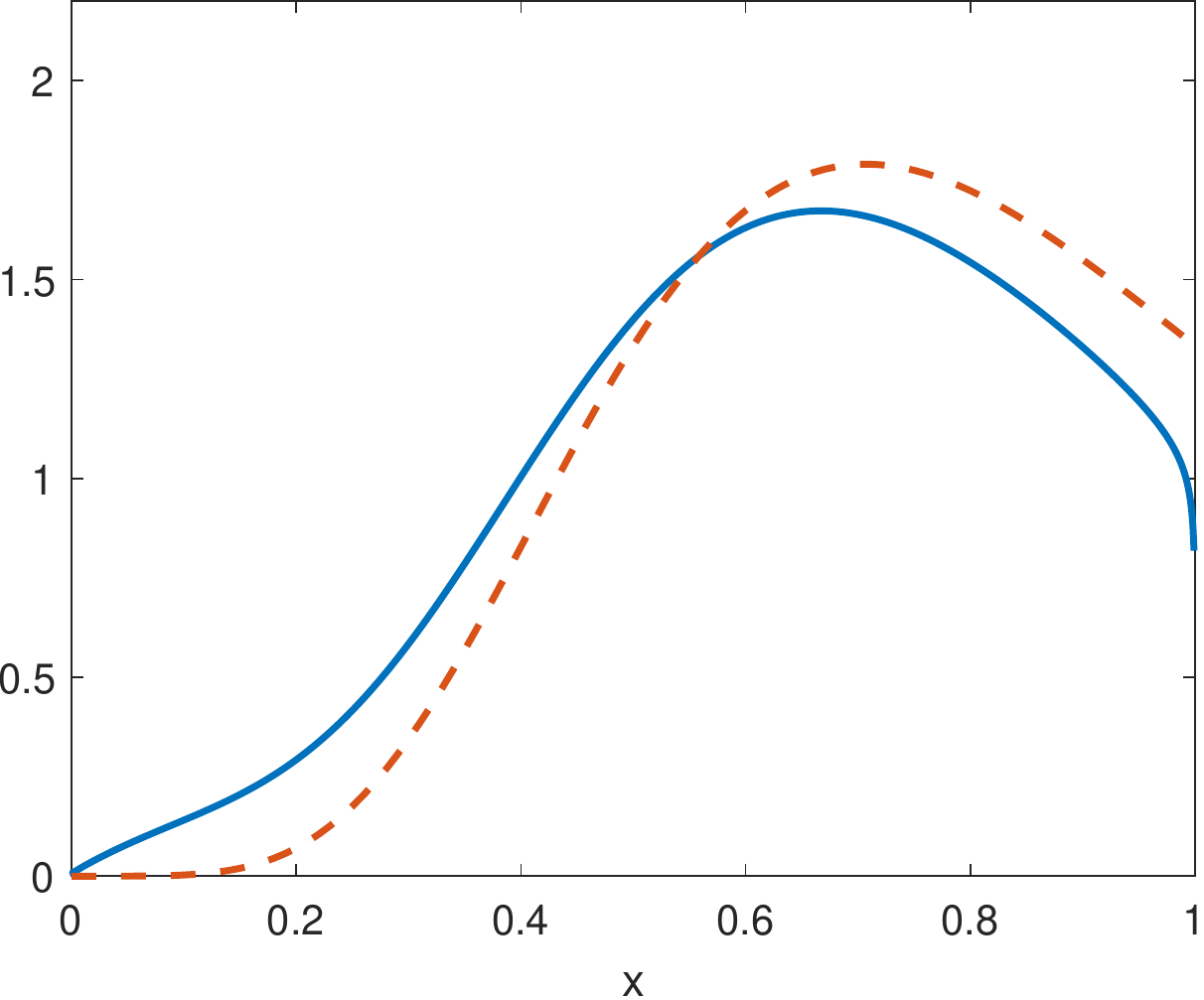}
\caption{}
\end{subfigure}
\begin{subfigure}[b]{\figwidth}
\includegraphics[width=\figwidth]{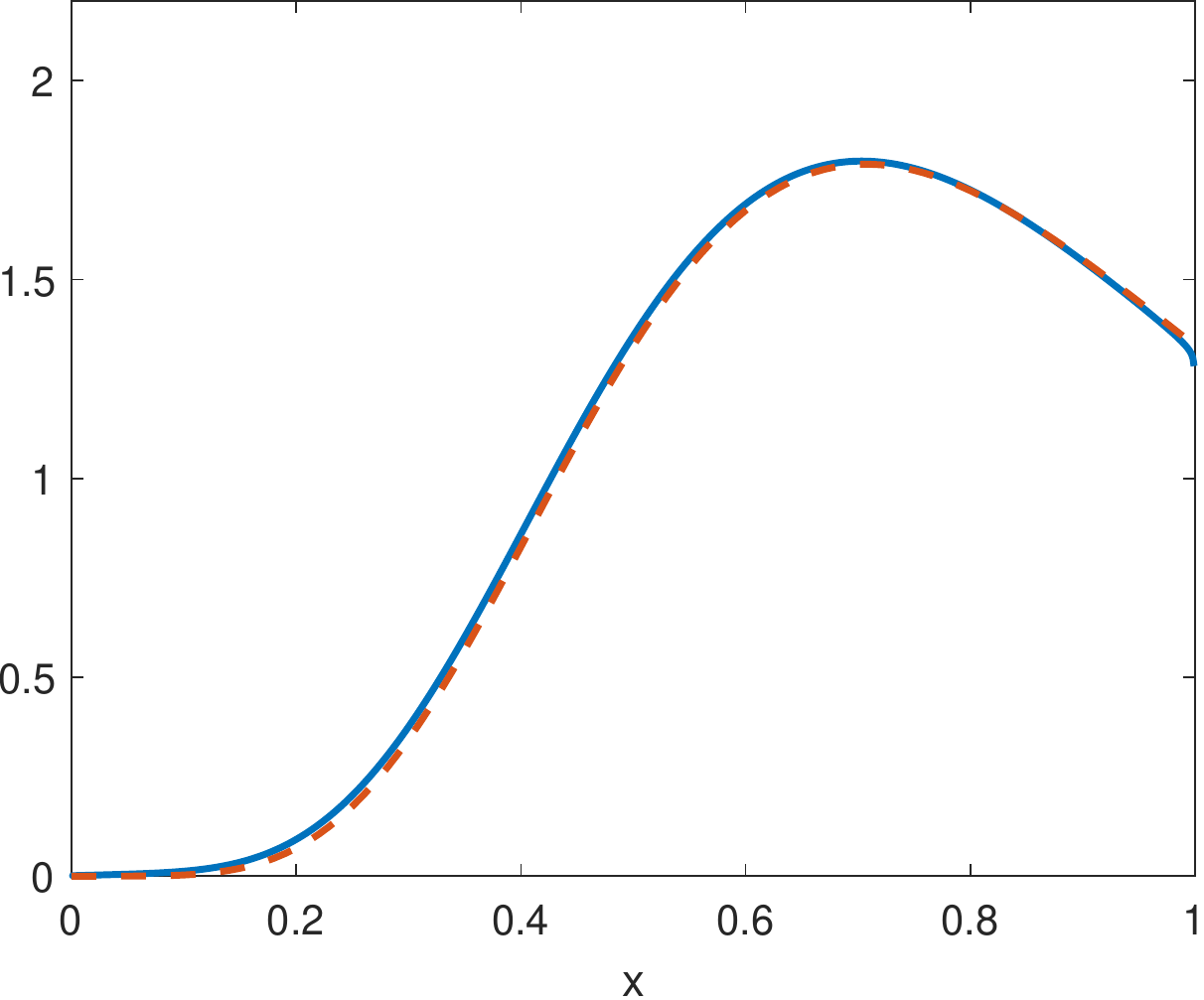}
\caption{}
\end{subfigure}
\caption{In this figure we provide simulations of the mean $V(t,x)$ (solid line) in Example~\ref{ex:pitch} for parameters
$q_0=4$, $q_1=2$, $\alpha_0=-0.5$, $\alpha_1=1$ at times  $t=0$ (a), $t=0.25$ (b), $t=0.7$ (c), $t= 2.5$ (d), $t= 5$ (e), and $t= 10$ (f). The dashed line graph  represents the  mean at large time $V^*(x)$ \label{fig3}
}
\end{figure}
\end{example}

Our last example treats the normal form of a supercritical Hopf bifurcation, see also  \cite{hurth}.

\begin{example}[Hopf bifurcation] Another commonly reported class of models in biology is one which exhibits a Hopf bifurcation.
The normal form for a supercritical Hopf bifurcation, after changing to polar coordinates $(\theta,r)$,
is
\begin{equation}\label{e:Hipf}
\begin{cases}
\theta'(t) = \omega +br^2(t),\\
r'(t) = \mu r(t) - r^3(t).
\end{cases}
\end{equation}
For $\mu < 0$ there is a single steady state $(\theta_*,r_*) = (0,0)$ while for $\mu > 0$ there is an unstable steady state at $(0,0)$ and a co-existing limit cycle with $r = \sqrt{\mu}$. In analogy to the previous cases, we take $\omega=\omega_i$ and $\mu=\mu_i$, $i=0,1$,  in \eqref{e:Hipf} with  $\mu_0<0$ and $\mu_1>0$. Let $E=\mathbb{S}^1\times (0,\infty)$, where $\mathbb{S}^1$ is the unit circle in $\mathbb{R}^2$.  To simplify the analysis we assume that $b=0$. If $\omega_0\neq \omega_1$ then the H\"ormander condition holds at every point $(\theta,r)\in E$. Note that the point $(0,\sqrt{\mu_1})$ is accessible from any point in $E$.  The asymptotic behaviour of the mean given by \eqref{e:meanP}  again depends on the sign of the parameter $\lambda$ in \eqref{d:lambda} with $b_i'(0)=\mu_i $, $i=0,1$. If $\lambda$ is positive then the mean at large time is equal to
  \begin{equation*}
    V^*(\theta,r)=\frac{1}{2\pi\kappa}\big(f_0(r)+f_1(r)\big)\mathbf{1}_{\mathbb{S}^1}(\theta),
  \end{equation*}
where $f_i$ has the form as in \eqref{pitch2} with $\alpha_i=\mu_i$.
On the other hand for $\lambda<0$ the mean of the process at large time is zero.

If $\omega=\omega_0= \omega_1$ then the angular variable $\theta$ is independent of the radial variable $r$ and it satisfies the same equation $\theta'(t)=\omega$ for each $i$. Thus, the process $(\theta(t),r(t),i(t))$ can be decomposed into two independent processes: $\theta(t)$ that is deterministic  and $(r(t),i(t))$ that behaves as the process in Example~\ref{ex:pitch}.

The asymptotic behaviour of the process in \eqref{def:process} is now different when $\lambda>0$. If we take the initial $g$ in \eqref{def:process} as the product of two marginal densities $g(\theta,r)=g_1(\theta)g_2(r)$ then the mean $V$ satisfies
\[
\lim_{t\to\infty}\int_E |V(t,\theta,r)-\frac{1}{\kappa}g_1(\theta-\omega t)\big(f_0(r)+f_1(r)\big)|d\theta dr=0,
\]
where $f_0$ and $f_1$ are as above.
\end{example}

\section{Second and higher order correlations}\label{sec:hmoment-eqn}

In this section we continue the study of the stochastic process \eqref{def:process} by looking at equations for correlations.
These are extensions of the moment equations considered in Section \ref{s:moment-eqn}.
We provide the full analysis only for second order correlations, but higher order cases are straightforward and can be easily obtained by similar considerations. We use some notation from the theory of tensor products, and for a brief summary of standard definitions used here see \ref{app:tensor}.

We start with the definition of second order correlations:
\begin{equation}\label{d:ViN}
C_i(t,x,y)=\mathbb{E}(\mathbf{1}_{\{i(t)=i\}} u(t,x)u(t,y)),\quad
x,y\in E,\ i\in I,\ t\ge 0.
\end{equation}
We will show that following equation holds:
\begin{equation}\label{e:VN}
\frac{\partial}{\partial t} C_i = (A_i\otimes \Id)C_i + (\Id\otimes A_i)C_i + \sum_{j\in I} q_{ji}C_j,\quad i\in I,
\end{equation}
where $A_i\otimes \Id$ and $\Id\otimes A_i$ are defined on functions $(x,y)\mapsto f(x,y)$  with $f\in L^1(E^2)$ as tensor products of operators $A_i$ and $\Id$.
Especially, if $A_i f(x)=-\dive (b_{i}(x)f(x))$ then we have
\begin{gather}\label{e:Adiv}
    (A_i\otimes \Id)f(x,y)=
-\dive (b_i(x)f(x,y)),\\
(\Id\otimes A_i)f(x,y)=
-\dive (b_i(y)f(x,y)).
\end{gather}

We consider equation \eqref{e:VN} in the space
$L^1(E^2\times I)=L^1(E^2\times I,\mathcal{B}(E^2\times I), \mu^2 )$, where $\mu^2$ is the product of two copies of
the measure $m$ on $E$ and the counting measure on $I$. We define the family of operators on
$L^1(E^2\times I)$   by
\begin{equation}\label{eqN}
(S(t)f )_i=\sum_{j \in I} \mathbb{E}_j (\mathbf{1}_{\{i(t)=i\}}U(t)\otimes U(t) f_{j}),\quad i \in I,
\end{equation}
for $f  =(f_j)_{j \in I} \in L^1(E^2\times I)$, $f_j \in L^1(E^2)$, $j \in I$, and where $U(t)$ is as in \eqref{d:Ut}.

\begin{theorem}\label{th2}
Assume conditions \eqref{a:1} and \eqref{a:2}. Then the family of operators $\{S(t)\}_{t \ge 0}$ defined in \eqref{eqN} is a stochastic semigroup on
$L^1(E^2\times I)$. The infinitesimal generator of this semigroup is the operator $A+Q^T$, where
\begin{equation}\label{Atensor}
A(f_i)_{i\in I} =\big( (A_i\otimes Id)f_i+(Id\otimes A_i)f_i \big)_{i\in I} \quad \text{and} \quad Q^T(f_i)_{i\in I}= (\sum_{j\in I} q_{ji}f_j)_{i \in I},
\end{equation}
with $f_i \in \mathcal{D}(A_i) \otimes \mathcal{D}(A_i), i\in I$.
\end{theorem}
\begin{proof}
Observe that
\[
U(t)\otimes U(t)=(P_{i(t_{N(t)})}(t-t_{N(t)})\otimes P_{i(t_{N(t)})}(t-t_{N(t)}))\circ  (U(t_{N(t)}) \otimes  U(t_{N(t)})),\quad t\ge 0.
\]
Let $i\in I$. Since $\{P_i(t)\}_{t\ge 0}$ is a stochastic semigroup on $L^1(E)$, we see that  $P_i(t) \otimes P_i(t)$ is a stochastic semigroup on $L^1(E^2)$, see Corollary \ref{c:stp}. Taking the injective tensor product space $B\widecheck{\otimes}B$ (see \ref{app:tensor} for the notation), we see that $T_i(t)\otimes T_i(t)$ is a $C_0$-semigroup on $B\widecheck{\otimes}B$ satisfying
\begin{equation}
\langle P_i(t) \otimes P_i(t)(f_1 \otimes f_2),h_1 \otimes h_2 \rangle=\langle f_1 \otimes f_2, T_i(t)\otimes T_i(t) (h_1 \otimes h_2) \rangle,
\end{equation}
where $f_1 \otimes f_2\in L^1(E) \otimes L^1(E) ,h_1 \otimes h_2 \in B \otimes B, t\ge 0$. Thus we have
\begin{equation}
\label{a:ad2}
\langle (P_i(t) \otimes P_i(t))f,h\rangle =\langle f, (T_i(t)\otimes T_i(t)) h \rangle, \quad t\ge 0,
\end{equation}
for all $f\in L^1(E) \otimes L^1(E)$ and $h\in  B \otimes B$. Since $L^1(E) \otimes L^1(E)$ is dense in $L^1(E^2)$ and $B \otimes B$ is dense in $B \widecheck{\otimes} B$, we conclude that \eqref{a:ad2} holds for all $f\in L^1(E^2)$ and $h\in B \widecheck{\otimes} B$.

The $\sigma$-algebra $\mathcal{B}(E^2)$ is generated by the $\pi$-system of  sets $\mathcal{C}\times\mathcal{C}=\{F_1\times F_2:F_1,F_2\in \mathcal{C}\}$. Given the set $F=F_1\times F_2$ we consider the sequence $h_n(x_1,x_2)=h_{1,n}(x_1)h_{2,n}(x_2)$, where $h_{1,n}$ and $h_{2,n}$ are sequences approximating the functions $\mathbf{1}_{F_1}$ and $\mathbf{1}_{F_2}$. Since  $h_n$ converges to $\mathbf{1}_F$, we see that condition \eqref{a:1} holds. Consequently,  Theorem~\ref{th1} implies that $\{S(t)\}_{t\ge 0}$ is a stochastic semigroup on $L^1(E^2\times I)$.
It follows from Proposition \ref{pro1} that for each $i \in I$ the generator of the semigroup $P_i(t) \otimes P_i(t)$ is the closure
of the operator $A_i \otimes \Id + \Id \otimes A_i$ defined on the core $\mathcal{D}(A_i) \otimes \mathcal{D}(A_i)$.
Thus the closure of the operator $A$ defined in \eqref{Atensor} is the generator
of the stochastic semigroup $(P_i(t) \otimes P_i(t))_{i \in I}.$ Hence, Theorem \ref{th1} implies that
the generator of the semigroup $\{S(t)\}_{t \ge 0}$ is the operator $A+Q^T$.
\end{proof}

Using Theorem \ref{th2} we obtain the following:
\begin{corollary} Assume conditions \eqref{a:1} and \eqref{a:2}. Let $S(t)$ be given by \eqref{eqN} and $u$ by
\eqref{d:pdm}. For each $g\in L^1(E^2)$ and $l\in I$ such that $u(0)=g$ and $i(0)=l$ we have $C_i(t,x,y)=(S(t)f )_i(x,y),$
where $C_i$ is as in \eqref{d:ViN} and $f =(f_j)_{j \in I}$ is of the form
\[
f_j=
\begin{cases}
g,\quad j=l, \\ 0,\quad j \neq l.
\end{cases}
\]
\end{corollary}

\begin{remark}
If for each $i\in I$ the semigroup $\{P_i(t)\}_{t\ge 0}$ is as in~\eqref{d:Pit} then the operator $A+Q^T$ from Theorem~\ref{th2}  is also the generator of a stochastic semigroup   induced by the
stochastic process $(x(t),y(t),i(t))$, where
$(x(t),y(t))$ satisfies the system of equations:
\begin{equation*}
\begin{cases}
x'(t)=b_{i(t)}(x(t)), \\ y'(t)=b_{i(t)}(y(t)).
\end{cases}
\end{equation*}
\end{remark}

\section{Conclusion}\label{sec:concl}

In this paper we introduced the concept of randomly switching stochastic semigroups. We investigated a
stochastic evolution equation in $L^1$ space. Such a regime could explain the source of stochasticity when
observing the evolution of some population driven by a common environmental stimulus. Next, we studied the
first moment of the stochastic evolution equation solutions and found the correspondence between this moment equation and a deterministic
system of Fokker-Planck type equations for the distributions of the process in Euclidean state space. We concluded that the mean of the process at large time can be expressed by the stationary solutions of a Fokker-
Planck type system providing that they exist. Similarly, we connected the mean of the process at large time with
sweeping property. We gave then some examples of the application of our results to biological models in which
the underlying dynamics display a variety of bifurcations  and provided numerical simulations for them. Finally, we studied second order correlations of
solutions of the stochastic evolution equation and we provided a rigorous way to extend our considerations to
correlations of higher order. Thus, this paper extends and justifies analytically the numerical results of Bressloff \cite{bressloff2017stochastic}.

The next step would be to show convergence in distribution of the infinite dimensional process $(u(t),i(t))$ to a stationary distribution. In particular examples connected with diffusion processes such convergence is known, see \cite{lawleyetal2015,cloez2015}. However, the results of \cite{lawleyetal2015,cloez2015} are not applicable to our stochastic semigroups $\{P_i(t)\}_{t\ge 0}$ on $L^1$ spaces because we have preservation of the norm while in these papers strict contraction on average was required. We hope to find in the future yet another approach that could be used for stochastic semigroups.

One possible future extension of this work is connected with addition of switching to stochastic PDEs driven by Gaussian noise or, more generally, by L\'evy noise, see \cite{peszat}. Another one could be related to randomly occurring phenomena in more complex systems like networks subjected to Markovian switching topology appearing in filtering problems as in \cite{liu2018} and \cite{ma2019}.  More careful recognition of these relations require further research.

\appendix\section{Tensor products}\label{app:tensor}  We recall some standard notation from the theory of tensor products~\cite{ryan}.
Let $\mathcal{X}_1$ and $\mathcal{X}_2$ be two Banach spaces of functions, i.e., either $\mathcal{X}_i$ is an $L^1$ space or it is a subspace of the space of all bounded measurable functions defined on a given set and equipped with the supremum norm. For $f_1\in \mathcal{X}_1$ and $f_2\in \mathcal{X}_2$ we identify the function $(x_1,x_2)\mapsto f_1(x_1)f_2(x_2)$ with the tensor  $f_1\otimes f_2$.  We define the tensor product space $\mathcal{X}_1\otimes \mathcal{X}_2$ as the set of all linear combinations of such tensors. The completion of the linear space $\mathcal{X}_1\otimes \mathcal{X}_2$  when equipped with  the projective norm
\[
\|h\|_{\pi}=\inf\{\sum_{k=1}^n\|f_k\|\|g_k\|: f_k\in \mathcal{X}_1, g_k\in \mathcal{X}_2, h=\sum_{k=1}^n f_k\otimes g_k\}
\]
is called the \emph{projective tensor product} of the spaces $\mathcal{X}_1$ and $\mathcal{X}_2$ and will be denoted by $\mathcal{X}_1\hat{\otimes} \mathcal{X}_2$.

It is known \cite[Chapter 2]{ryan} that $L^1(E_1,\mathcal{E}_1,m_1)\hat{\otimes} L^1(E_2,\mathcal{E}_2,m_2)$ is isometrically isomorphic with $L^1(E_1\times E_2,\mathcal{E}_1\times \mathcal{E}_2,m_1\times m_2)$.
If instead we consider $\mathcal{X}_1\otimes \mathcal{X}_2$ with  the injective norm
\[
\|h\|_{\varepsilon}=\sup\{|(\gamma_1\otimes\gamma_2)(h)|: \gamma_i\in \mathcal{X}_i^{*}, \|\gamma_i\|\le 1 \},
\]
where $\mathcal{X}_i^{*}$ is the dual of $\mathcal{X}_i$ and
\[
(\gamma_1\otimes\gamma_2)(h)=\sum_{k=1}^n \gamma_1(f_k)\gamma_2(g_k)\quad \text{for } h=\sum_{k=1}^n f_k\otimes g_k,
\]
then the completion  of $\mathcal{X}_1\otimes \mathcal{X}_2$ is called the \emph{injective tensor product} and it will be denoted by $\mathcal{X}_1\widecheck{\otimes} \mathcal{X}_2$. In particular, if  $C(E_i)$ is the space of continuous functions on a compact space $E_i$ then $C(E_1)\widecheck{\otimes}C(E_2)$ is the space $C(E_1\times E_2)$, by \cite[Section 3.2]{ryan}.
Note that if $\mathcal{Y}_i$ is a closed linear subspace of the Banach space $\mathcal{X}_i$ then $\mathcal{Y}_1\widecheck{\otimes} \mathcal{X}_2$ and $\mathcal{X}_1\widecheck{\otimes} \mathcal{Y}_2$ are closed liner subspaces of $\mathcal{X}_1\widecheck{\otimes} \mathcal{X}_2$.

Given two linear and bounded operators $S_i\colon \mathcal{X}_i\to \mathcal{X}_i$ the linear mapping $S_1\otimes S_2\colon \mathcal{X}_1\otimes \mathcal{X}_2\to \mathcal{X}_1\otimes \mathcal{X}_2$  defined by
\[
(S_1\otimes S_2)(f_1\otimes f_2)=S_1(f_1)\otimes S_2(f_2)
\]
has a continuous extension to tensor product spaces.
We will use the following result from \cite[Section A-I.3, Proposition]{nagel86-0}:

\begin{proposition}\label{pro1}
Let $\{S_1(t) \}_{ t \ge 0}$ and $\{S_2(t) \}_{ t \ge 0}$ be $C_0$-semigroups on some Banach spaces $\mathcal{X}_1$, $\mathcal{X}_2$ and let the operators
$(A_1,\mathcal{D}(A_1))$, $(A_2,\mathcal{D}(A_2))$ be their generators. Then the family
\begin{equation}
\{S_1(t) \otimes S_2(t)\}_{ t \ge 0}
\end{equation}
is a $C_0$-semigroup on both  projective and injective tensor products of  $\mathcal{X}_1$ and $\mathcal{X}_2$.
The closure of
\begin{equation}
A_1 \otimes \Id  + \Id \otimes A_2,
\end{equation}
defined on the core $\mathcal{D}(A_1) \otimes \mathcal{D}(A_2)$, is its generator.
\end{proposition}

\begin{corollary}\label{c:stp}
If  $\{S_1(t) \}_{ t \ge 0}$ and $\{S_2(t) \}_{ t \ge 0}$ are stochastic semigroups on the spaces $L^1(E_1,\mathcal{E}_1,m_1)$ and $L^1(E_2,\mathcal{E}_2,m_2)$, respectively,  then  $\{S_1(t) \otimes S_2(t)\}_{ t \ge 0}$ is a stochastic semigroup on $L^1(E_1\times E_2,\mathcal{E}_1\times \mathcal{E}_2,m_1\times m_2)$.
\end{corollary}
\begin{proof}
For $f_i\in L^1(E_i,\mathcal{E}_i,m_i)$ we have
\[
\int_{E_1\times E_2}(S_1(t)\otimes S_2(t))(f_1\otimes f_2) d(m_1\times m_2) =\int_{E_1}S_1(t)f_1dm_1\int_{E_2}S_2(t)f_2dm_2.
\]
This implies that the operator $S_1(t)\otimes S_2(t)$ preserves the integral. It is easy to see that $S_1(t)\otimes S_2(t)$ is a positive operator, completing the proof.
\end{proof}

\section*{Acknowledgements}

We would like to thank Ryszard Rudnicki and Rados\l{}aw Wieczorek for helpful discussions.  The authors are especially appreciative of the comments of three anonymous referees that materially improved the presentation.    This research was supported in part by the Natural Sciences and Research Council of Canada (NSERC) and  the Polish NCN grant  2017/27/B/ST1/00100.


\end{document}